\documentclass{amsart}

\usepackage[ngerman,english]{babel}
\usepackage[applemac]{inputenc}
\usepackage{amsmath,amsthm,amssymb,amsfonts,nicefrac}
\usepackage[all]{xy}
\usepackage{amsmath}
\usepackage{graphicx}
\usepackage{textcomp}
\usepackage{color}
\usepackage{graphpap}
\usepackage{float}

\renewcommand{\ge}{\geqslant}

\newcommand{\N}{\ensuremath{\mathbb{N}}}

\newcommand{\F}{\ensuremath{\mathbb{F}}}

\newcommand{\R}{\ensuremath{\mathbb{R}}}
\newcommand{\C}{\ensuremath{\mathbb{C}}}
\newcommand{\PP}{\ensuremath{\mathbb{P}}}

\newcommand{\FF}{\ensuremath{\mathbb{F}}}
\DeclareMathOperator{\Aut}{Aut}
\DeclareMathOperator{\Bir}{Bir}
\DeclareMathOperator{\GL}{GL}
\DeclareMathOperator{\Bp}{sBp}
\DeclareMathOperator{\TAut}{TAut}
\DeclareMathOperator{\Id}{Id}

\DeclareMathOperator{\SL}{SL}
\def\dashmapsto{\mapstochar\dashrightarrow}

\newtheorem{Thm}{Theorem}[section]
\newtheorem*{ThmA*}{Theorem A}
\newtheorem*{ThmB*}{Theorem B}
\newtheorem{Thm*}{Theorem} 
\newtheorem{Cor}[Thm]{Corollary}
\newtheorem{Lem}[Thm]{Lemma}
\newtheorem{Prop}[Thm]{Proposition}
\theoremstyle{definition}
\newtheorem{Def}[Thm]{Definition}
\newtheorem{Rmk}[Thm]{Remark}

\newtheorem{Ex}[Thm]{Example}

\title{The Cremona group of the plane is compactly presented}
\author{Susanna Zimmermann}
\thanks{The author gratefully acknowledges support by the Swiss National Science Foundation Grant ``Birational geometry'' PP00P2\_153026 /1}

\begin{document}

\maketitle
\thispagestyle{empty}

\setcounter{tocdepth}{1}
\begin{abstract}
This article shows that the Cremona group of the plane is compactly presented. To do this we prove that it is a generalised amalgamated product of three of its algebraic subgroups (automorphisms of the plane and Hirzebruch surfaces) divided by one relation.
\end{abstract}

\tableofcontents

\section{Introduction}
Let $k$ be a field. The Cremona group $\Bir(\PP^n)$ is the group of birational transformations of the projective space $\PP^n_k=\PP^n$. It corresponds to a very intensively studied topic in algebraic geometry (see \cite{Ser08,Des,Cant13} and references therein.)

A birational transformation of $\PP^n$ is simply a birational change of coordinates, so $\Bir(\PP^n)$ is a natural generalisation of $\Aut(\PP^n)= \mathrm{PGL}_{n+1}(k)$ and in many aspects the Cremona group behaves like semi-simple groups, but also in many aspects it does not. Some analogies between the Cremona groups and semi-simple groups have been presented by J.-P. Serre in the 1000th Bourbaki seminar \cite{Ser08}, and by S.~Cantat in \cite{Cant11}. 

Endowed with the Euclidean topology, constructed in \cite{BF13}, the Cremona group becomes a Hausdorff topological group (for $k$ a local field). For $k=\C$ and $k=\R$ the restriction to its subgroup $\mathrm{PGL}_{n+1}(k)$ of linear coordinate changes of $\PP^n$ is the Euclidean topology. This not only opens the path to study the geometric properties of the Cremona group coming from the Euclidean topology but also presents the opportunity to study the Cremona group from the point of view of geometric group theory and rises the question of analogies to Lie groups. 

In this article we will present one of these analogies, namely the property of being compactly presented (see Definition~\ref{def cp}). 

Lets take a closer look at the Cremona group endowed with the Euclidean topology:

The group $\Bir(\PP^1_{\C})=\mathrm{PGL}_2(\C)$ is compactly presented by any neighbourhood of $1$ because it is a connected complex algebraic group (see for example \cite[Satz~3.1]{Ab}). 

For $n\geq 2$ and $k$ any local field, the group $\Bir(\PP^n_k)$ is not locally compact \cite[Lemma~5.15]{BF13}, though the topology is the inductive topology given by the family of closed sets $\Bir(\PP^n)_{\leq d}=\{f\in\Bir(\PP^n)\mid \deg(f)\leq d\}$, which are locally compact \cite[Proposition~2.10, Lemma~5.4]{BF13}. Furthermore, any compact subset of $\Bir(\PP^n)$ is of bounded degree.

For $n\geq 3$, the group $\Bir(\PP^n_{\C})$ is not compactly generated \cite[Lemma~5.17]{BF13}, hence not compactly presented.

The group $\Bir(\PP^2_{\C})$ is generated by $\Aut(\PP^2_{\C})=\mathrm{PGL}_2(\C)$ and the standard quadratic transformation $\sigma: [x:y:z]\dashmapsto[yz:xz:xy]$ \cite{Cas}. Its  subgroup $\Aut(\PP^2_{\C})$, being a connected complex algebraic group, is compactly presented by any neighourhood of $1$. Hence $\Bir(\PP^2_{\C})$ is compactly generated by any compact neighbourhood of 1 in $\Aut(\PP^2_{\C})$ and $\sigma$.

The aim of this article is to show that, even though it is neither an algebraic group nor locally compact, $\Bir(\PP^2_{\C})$ is moreover compactly presented: 

\begin{ThmA*}[(Corollary~\ref{cor cpt pres 2})]\label{Thm A}
Endowed with the Euclidean topology the Cremona group $\Bir(\PP^2_{\C})$ is compactly presented by $\{\sigma\}\cup K$ where $K$ is any compact neighbourhood of $1$ in $\Aut(\PP^2_{\C})$ and $\sigma:[x:y:z]\dashmapsto[yz:xz:xy]$ is the standard involution of $\PP^2_{\C}$.
\end{ThmA*} 

For algebraically closed fields, the generating sets and generating relations of $\Bir(\PP^2)$ have been studied throughoutly: The famous Noether-Castelnuovo theorem \cite{Cas} states that if $k$ is algebraically closed then $\Bir(\PP^2)$ is generated by $\Aut(\PP^2)$ and the standard quadratic involution $\sigma:[x:y:z]\dashmapsto[yz:xz:xy]$, i.e. the generating set is the union of two complex linear algebraic groups. 

A presentation was given in \cite{Giz} where the generating set consists of all quadratic transformations of $\PP^2$ and the generating relations are of the form $q_1q_2q_3=1$ where $q_i$ are quadratic transformations. Another presentation was given in \cite{Isk} where it is shown that $\Bir(\PP^1\times\PP^1)$ (isomorphic to $\Bir(\PP^2)$) is the amalgamated product of $\Aut(\PP^1\times\PP^1)$ and the de Jonquières group of birational maps of $\PP^1\times\PP^1$ preserving the first projection along their intersection modulo one relation. In \cite{Bla12} a similar result is presented; the group $\Bir(\PP^2)$ is the amalgamated product of $\Aut(\PP^2)$ and the de Jonquières group $J_{[1:0:0]}$ of birational maps of $\PP^2$ preserving the pencil of lines through $[1:0:0]$ along their intersection modulo one relation. 

Since neither $\Aut(\PP^2)=\mathrm{PGL}_3(k)$ nor the set of quadratic transformations nor the de Jonquières group are compact in the Euclidean topology, these presentations yield no compact presentation but at least the latter two yield a bounded presentation (the length of the generating relations are universally bounded). 

In \cite{W}, using \cite{Isk}, a presentation of $\Bir(\PP^2)$ is given by the generalised amalgamated product of $\Aut(\PP^2)$, $\Aut(\PP^1\times\PP^1)$ and $J_{[0:1:0]}$ (as subgroups of $\Bir(\PP^2)$) along their pairwise intersection, where $\Aut(\PP^1\times\PP^1)$ is viewed as a subgroup of $\Bir(\PP^2)$ via a birational map $\PP^2\dashrightarrow\PP^1\times\PP^1$ given by the pencils of lines through $[0:1:0]$ and $[1:0:0]$. Again, since $J_{[0:1:0]}$ is not compact, this does not yield a compact but only a bounded presentation but it gives rise to the following idea, which is the key step in the proof of Theorem A: 

\begin{ThmB*}[(Theorem~\ref{amalg})]\label{Thm B}
Let $k$ be algebraically closed. Then the Cremona group $\Bir(\PP^2)$ is isomorphic to the amalgamated product of $\Aut(\PP^2)$, $\Aut(\FF_2)$, $\Aut(\PP^1\times\PP^1)$ $($as subgroups of $\Bir(\PP^2)$$)$ along their pairwise intersection in $\Bir(\PP^2)$ modulo the relation $\tau_{13}\sigma\tau_{13}\sigma$, where $\tau_{13}\in\Aut(\PP^2)$ is given by $\tau_{13}\colon[x:y:z]\mapsto[z:y:x]$. 
\end{ThmB*}

Here the inclusion of $\Aut(\PP^1\times\PP^1)$ into $\Bir(\PP^2)$ is the same as before, $\FF_2$ is the second Hirzebruch surface and the inclusion of $\Aut(\FF_2)$ into $\Bir(\PP^2)$ is given by a birational map $\PP^2\dashrightarrow\FF_2$ given by the system of lines through $[1:0:0]$ and the point infinitely near corresponding to the tangent direction $\{y=0\}$. 

The method used to prove Theorem B is, like in \cite{Bla12} and \cite{Isk}, to study linear systems and their base-points. The difference here is that our maps have bounded degree which rigidifies the situation and changes the possibilities for simplifications. The proof of Theorem B does not use \cite{W}.

For $k=\C$, the three groups $\Aut(\PP^2)$, $\Aut(\FF_2)$, $\Aut(\PP^1\times\PP^1)$ are locally compact algebraic groups. Using this and Theorem B we prove Theorem A.\\

The plan of the article is as follows: 

In Section~\ref{1} and Section~\ref{2} we give basic definitions and results on $\Aut(\PP^1\times\PP^1)$ and $\Aut(\FF_2)$. Section~\ref{3} is devoted to relations in the generalised amalgamated product of $\Aut(\PP^2)$, $\Aut(\FF_0)$, $\Aut(\FF_2)$ modulo the relation $\tau_{13}\sigma\tau_{13}\sigma$. These are the backbone of the proof of Theorem B, which will be given in Section~\ref{4}. In Section~\ref{5} we visit some facts about compactly presented groups and then finally prove Theorem A. 

In Section~\ref{1} to \ref{4} we work over any algebraically closed field $k$ and Section~\ref{5} restricts to $k=\C$.\\

The author would like to thank Jérémy Blanc, Pierre de la Harpe and Yves de Cornulier for the interesting and helpful discussions.

\section{Description of $\Aut(\PP^1\times\PP^1)$ and $\Aut(\FF_2)$ inside the Cremona group}\label{1}

This section is devoted to the description of the subgroups $\Aut(\PP^1\times\PP^1)$ and $\Aut(\FF_2)$ of $\Bir(\PP^2)$. 

Remember that the $n$-th Hirzebruch surface $\FF_n$, $n\in\N$, is given by
$$\FF_n=\{([x:y:z],[u:v])\in\PP^2\times\PP^1\mid yv^n=zu^n\}.$$
Observe that $\FF_0= \PP^1\times\PP^1$ and that $\FF_1$ is isomorphic to the blow-up of one point in $\PP^2$.

Consider the birational maps $\varphi_0:\PP^2\dashrightarrow \FF_0$ and $\varphi_2:\PP^2\dashrightarrow\FF_2$ given as follows: The map $\varphi_0$ is given by the blow-up of the points $[1:0:0]$ and $[0:1:0]$ followed by the contraction of the line passing through them. The map $\varphi_2$ is given by the blow up of $[1:0:0]$ and the point infinitely near $[1:0:0]$ lying on the strict transform of $\{y=0\}$ followed by the contraction of the strict transform of $\{y=0\}$. The birational maps $\varphi_0$ and $\varphi_2$ are only defined up to automorphism of $\FF_0$ and $\FF_2$. 
They induce homomorphisms of groups
\begin{align*}&\Aut(\FF_0)\rightarrow\Bir(\PP^2),\quad \psi\mapsto \varphi_0^{-1}\psi\varphi_0\\ &\Aut(\FF_2)\rightarrow\Bir(\PP^2),\quad \psi\mapsto\varphi_2^{-1}\psi\varphi_2\end{align*} 
whose image is uniquely determined by the choice of points blown-up in $\PP^2$. We will denote the image of $\Aut(\FF_i)$ also by $\Aut(\FF_i)$ for $i=0,2$ since no confusion occurs. 

\begin{Rmk}[and Notation]
\label{def inverse}
\begin{enumerate}
\item\label{rmk trans} We can check that 
$$\PP^2\dashrightarrow\FF_0,\quad[x:y:z]\dashmapsto([x:z],[y:z])$$
with inverse $([u_0:u_1],[v_0:v_1])\dashmapsto[u_0v_1:v_0u_1:u_1v_1]$, and \\
$$\PP^2\dashrightarrow\FF_2,\quad[x:y:z]\dashmapsto([xy:y^2:z^2],[y:z])$$ 
with inverse $([u:v:w],[a:b])\dashmapsto[ua:va:vb]$ are examples for $\varphi_0$ and $\varphi_2$. 
\item For $i=0,2$, the map $(\varphi_i)^{-1}$ has exactly one base-point, which we denote by $p_i$. 
\item\label{rmk pt} The image of the linear system of lines of $\PP^2$ by $\varphi_i$ has a unique base-point, namely $p_i$.
\item We denote by $C_1$ the curve of self intersection 0 in $\FF_0$ which is contracted by $(\varphi_0)^{-1}$ onto $[1:0:0]$ and by $C_2$ the curve of self intersection 0 which is contracted onto $[0:1:0]$. Remark that $p_0=\varphi_0(\{z=0\})$ and that $\{p_0\}=C_1\cap C_2$.
\item We denote by $E$ the exceptional curve of self intersection -2 in $\FF_2$. It is contracted onto $[1:0:0]$ by $(\varphi_2)^{-1}$. Denote by $C$ the curve of self intersection 0 in $\FF_2$ which is contracted by $(\varphi_2)^{-1}$ onto the point infinitely near $[1:0:0]$ corresponding to the tangent $\{y=0\}$. Remark that $p_2=\varphi_2(\{y=0\})$ and $p_2\in C\setminus E$.
\item\label{rmk inters} Let $L\subset\PP^2$ be a general line. Then $C_j\cdot\varphi_0(L)=1$, $j=1,2$, and $C\cdot\varphi_2(L)=1$, $E\cdot\varphi_2(L)=0$. 

\end{enumerate}
\end{Rmk}

The following picture illustrates for $i=0,2$ the transformation $(\varphi_i)^{-1}\psi_i\varphi_i$ where $\psi_i$ is some automorphism of $\FF_i$. At the same time it shows the blow-up diagram of $(\varphi_i)^{-1}\psi_i\varphi_i$.

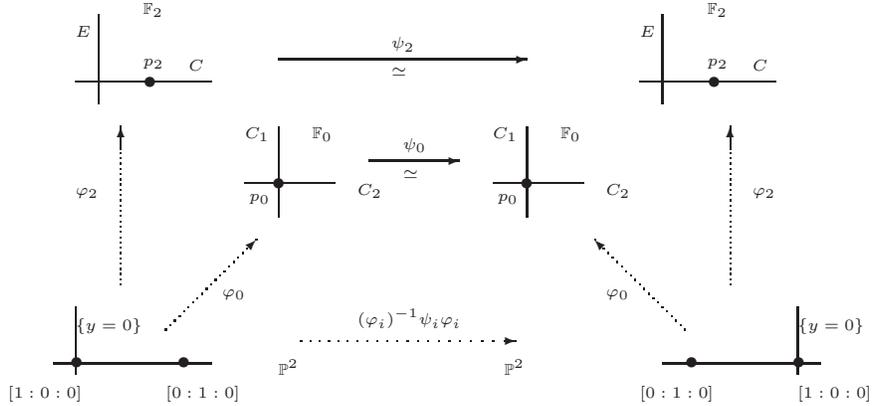
\begin{figure}[H]
\begin{center}
\begin{picture}(125,135)
\setlength{\unitlength}{3mm}
\multiput(-10,0.5)(27,0){2}{\line (1,0){7}}
\multiput(-9,0)(32,0){2}{\line(0,1){3}}
\multiput(-9,2)(32,0){2}{\tiny$\{y=0\}$}
\multiput(-9.25,0.25)(32,0){2}{\unboldmath$\bullet$}
\multiput(-12,-1)(35,0){2}{\tiny$[1:0:0]$}
\multiput(-4.5,0.25)(22.5,0){2}{\unboldmath$\bullet$}
\multiput(-5,-1)(21,0){2}{\tiny$[0:1:0]$}
\multiput(0,0)(10,0){2}{\tiny$\PP^2$}
\qbezier[20](1,1.5)(5,1.5)(10,1.5)
\put(10,1.5){\vector(1,0){0.5}}
\put(3.5,2.25){\tiny$(\varphi_i)^{-1}\psi_i\varphi_i$}
\multiput(1.5,10.5)(11,0){2}{\tiny$\FF_0$}
\qbezier[20](-5,2)(-3,4)(-1,6)
\put(-2,5){\vector(1,1){1}}
\qbezier[20](18,2)(16,4)(14,6)
\put(15,5){\vector(-1,1){1}}
\multiput(-2.5,3.5)(17,0){2}{\tiny$\varphi_0$}
\multiput(0,7)(11,0){2}{\line(0,1){4}}
\multiput(-1.5,10.5)(11,0){2}{\tiny$C_1$}
\multiput(-1.5,8.5)(11,0){2}{\line(1,0){4}}
\multiput(3.5,8)(11,0){2}{\tiny$C_2$}
\multiput(-0.3,8.2)(11,0){2}{\unboldmath$\bullet$}
\multiput(-1.3,7.7)(11,0){2}{\tiny$p_0$}
\put(4,9.5){\vector(1,0){4}}
\put(5.5,10){\tiny$\psi_0$}
\put(5.5,8.75){\tiny$\simeq$}
\multiput(-6,16)(25,0){2}{\tiny$\FF_2$}
\qbezier[30](-7,4)(-7,7.5)(-7,11)
\put(-7,10){\vector(0,1){1}}
\qbezier[30](20,4)(20,7.5)(20,11)
\put(20,10){\vector(0,1){1}}
\multiput(-9,8)(30,0){2}{\tiny$\varphi_2$}
\multiput(-8,12)(25,0){2}{\line(0,1){4}}
\multiput(-9,15)(25,0){2}{\tiny$E$}
\multiput(-9,13)(25,0){2}{\line(1,0){6}}
\multiput(-4,13.5)(25,0){2}{\tiny$C$}
\multiput(-6,12.7)(25,0){2}{\unboldmath$\bullet$}
\multiput(-6,13.7)(25,0){2}{\tiny$p_2$}
\put(0,14){\vector(1,0){11}}
\put(5,14.5){\tiny$\psi_2$}
\put(5,13.25){\tiny$\simeq$}
\end{picture}
\end{center}
\caption{The transformation $(\varphi_i)^{-1}\psi_i\varphi_i$ for $i=0,2$}
\label{blow-up}
\end{figure}

Consider the birational transformations of $\PP^2$ given by
\begin{align*}&\sigma_{1}:[x:y:z]\dashmapsto[-xy+z^2:y^2:yz]\\&\sigma_{2}:[x:y:z]\dashmapsto[xy:z^2:yz]\\&\sigma_3:[x:y:z]\dashmapsto[yz:xz:xy]\end{align*} 
They are three quadratic involutions of $\PP^2$ with respectively exactly one, two and three proper base-points in $\PP^2$. The map $\sigma_3$ is usually referred to as standard quadratic involution of $\PP^2$. 

The map $\sigma_3$ has base-points $[1:0:0],[0:1:0],[0:0:1]$, the map $\sigma_2$ has base-points $[1:0:0],[0:1:0]$ and the point $p$ infinitely near $[1:0:0]$ corresponding to the direction $\{y=0\}$, and the map $\sigma_1$ has base-points $[1:0:0],p,q$ where $q$ is a point infinitely near $p$ not contained in the intersection of the strict transform of the exceptional divisor of $[1:0:0]$.

\begin{Rmk}\label{rmk linsyst}
For any quadratic map $\tau\in\Bir(\PP^2)$ we can find $i=1,2,3$ and $\alpha,\beta\in\Aut(\PP^2)$ such that $\alpha$ sends the base-points of $\tau$ onto the base-points of $\sigma_i$ and $\beta$ sends the base-points of $\tau^{-1}$ onto the base-points of $\sigma_i$. We can then write $\tau=\beta^{-1}\sigma_i\alpha$. \cite[\S2.1 and \S2.8]{Alb02}

It follows that the linear system of $\tau$ is the image of the linear system of $\sigma_i$ by $\alpha^{-1}$, and that $\tau$ and $\sigma_i$ have the same amount of proper base-points in $\PP^2$. Since $\sigma_1,\sigma_2,\sigma_3$ have respectively one, two and three proper base-points in $\PP^2$ the amount of proper base-points of $\tau$ determines $i$. 
\end{Rmk}

The following is the description of the groups $\Aut(\FF_0)$ and $\Aut(\FF_2)$ as subgroups of $\Bir(\PP^2)$ given by the above inclusions:

\begin{Lem}\label{lem auto linear}
\begin{enumerate}
\item\label{11} For $i=0,2$ the group $\mathcal{A}_i:=\Aut(\FF_i)\cap\Aut(\PP^2)$ is the group of automorphisms of $\PP^2$ fixing the set of base-points of $\varphi_i$, i.e. the set $\{[1:0:0],[0:1:0]\}$ if $i=0$, and the point $[1:0:0]$ and the line $\{y=0\}$ if $i=2$. 

For each $i\in \{0,2\}$, $\mathcal{A}_i$ corresponds via $\varphi_i$ to the set of automorphisms of $\FF_i$ that fix $p_i$. 
\item\label{13} The set $\Aut(\FF_0)\setminus\Aut(\PP^2)$ consists of all elements of the form $\beta\sigma_i\alpha$, where $i=2,3$ and $\alpha,\beta\in\Aut(\FF_0)\cap\Aut(\PP^2)$. 
\item\label{13.5} The set $\mathcal{A}_0\cup\mathcal{A}_0\sigma_2\mathcal{A}_0$ corresponds via $\varphi_0$ to the set of automorphisms of $\FF_0$ sending $p_0$ into $C_1\cup C_2$. 
\item\label{14} The set $\Aut(\FF_2)\setminus\Aut(\PP^2)$ consists of all elements of the form $\beta\sigma_i\alpha$, where $i=1,2$ and $\alpha,\beta\in\Aut(\FF_2)\cap\Aut(\PP^2)$.
\item\label{15} The set $\mathcal{A}_2\cup\mathcal{A}_2\sigma_1\mathcal{A}_2$ corresponds via $\varphi_2$ to the set of automorphisms of $\F_2$ that send $p_2$ into $C$.
\end{enumerate}
\end{Lem}

\begin{proof} 
For $i=0,2$ let $\psi_i$ be an automorphism of $\FF_i$ and consider the following commutative diagram
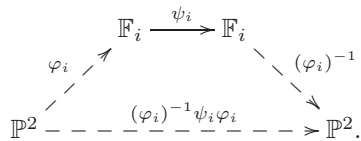
\[\xymatrix{&\FF_i\ar[r]^{\psi_i}&\FF_i\ar@{-->}[dr]^{(\varphi_i)^{-1}}& \\ \PP^2\ar@{-->}[rrr]^{(\varphi_i)^{-1}\psi_i\varphi_i}\ar@{-->}[ur]^{\varphi_i}&&&\PP^2. }\] 

\ref{11}: The map $(\varphi_i)^{-1}\psi_i\varphi_i$ is an automorphism if and only if it does not have any base-points, which is equivalent to $\psi_i$ preserving the union of the curves contracted by $(\varphi_i)^{-1}$ and fixing the point $p_i$ blown-up by $(\varphi_i)^{-1}$. This is equivalent to $(\varphi_i)^{-1}\psi_i\varphi_i$ being an automorphism preserving the set of base-points of $\varphi_i$.

\ref{13} to \ref{15}: Let $\Delta$ be the linear system of lines in $\PP^2$. We will determine the linear system $(\varphi_i)^{-1}\psi_i\varphi_i(\Delta)$. Note that \ref{11} shows that $(\varphi_i)^{-1}\psi_i\varphi_i(\Delta)=\Delta$ if and only if $\psi_i$ preserves $p_i$ and the union of lines contracted by $(\varphi_i)^{-1}$, which is equivalent to $\psi_i$ fixing $p_i$.

Assume that $\psi_i(p_i)\neq p_i$ holds. From this it follows that $(\varphi_i)^{-1}\psi_i\varphi_i$ has at least one and at most three base-points, hence is a quadratic map. In particular, $(\varphi_i)^{-1}\psi_i\varphi_i(\Delta)$ is a linear system of conics. 

\ref{13} and \ref{13.5}: If $i=0$, we can check that $\sigma_2,\sigma_3$ are elements of $\Aut(\FF_0)$. In fact, if we take $\varphi_0$ as in Remark~\ref{def inverse}~\ref{rmk trans} they are given by the automorphisms $([u_0:u_1],[v_0:v_1])\dashmapsto[u_0v_1:v_0u_1:u_1v_1]$ and $([u_1:v_1],[u_2:v_2])\dashmapsto([v_1:u_1],[v_2:u_2])$ respectively. It follows that the set $\mathcal{A}_0\sigma_2\mathcal{A}_0\cup\mathcal{A}_0\sigma_3\mathcal{A}_0$ is contained in $\Aut(\FF_0)$. 

A general element of $\psi_0\varphi_0(\Delta)$ intersects each $C_j$ in exactly one point different from $p_0$ (Remark~\ref{def inverse}~\ref{rmk pt},\ref{rmk inters}), which means that $[1:0:0],[0:1:0]$ are base-points of the linear system of conics $(\varphi_0)^{-1}\psi_0\varphi_0(\Delta)$. The third base-point corresponds via $\varphi_0$ to the point $\psi_0(p_0)$. In particular, it is infinitely near to $[1:0:0]$ (resp. $[0:1:0]$) if and only if $\psi_0(p_0)\in C_1$ (resp. $\psi_0(p_0)\in C_2$). By Remark~\ref{rmk linsyst} we can write $(\varphi_0)^{-1}\psi_0\varphi_0=\beta\sigma_j\alpha$ for some $j=2,3$ and $\alpha,\beta\in\Aut(\PP^2)$, where $\alpha,\beta^{-1}$ respectively send the base-points of $(\varphi_0)^{-1}\psi_0\varphi_0$, $(\varphi_0)^{-1}(\psi_0)^{-1}\varphi_0$ onto the base-points of $\sigma_j$. If $j=2$, it follows that $\alpha,\beta$ fix the set $\{[1:0:0],[0:1:0]\}$. If $j=3$, we have $(\beta\theta)\sigma_3(\theta\alpha)=\beta\sigma_3\alpha$ for any permutation $\theta$ of coordinates $x,y,z$, hence we can assume that $\alpha$ fixes the set $\{[1:0:0],[0:1:0]\}$ and it follows that $\alpha\in\mathcal{A}_0$. Since $\sigma_2,\sigma_3\in\Aut(\FF_0)$, it follows that $\beta\in\mathcal{A}_0$. 

Note that $(\varphi_0)^{-1}\psi_0\varphi_0$ has an infinitely near base-point if and only if $\psi_0(p_0)\in (C_1\cup C_2)\setminus \{p_0\}$. 

\ref{14} and \ref{15}: If $i=2$, we can check that $\sigma_1,\sigma_2\in\Aut(\FF_2)$. In fact, if we take $\varphi_2$ as in Remark~\ref{def inverse}~\ref{rmk trans} then they are given by the automorphisms $([u:v:w],[a:b])\mapsto([-u+w:v:w],[a:b])$ and $([u:v:w],[a:b])\mapsto([u:w:v],[b:a])$ respectively. It follows that $\mathcal{A}_2\sigma_1\mathcal{A}_2\cup\mathcal{A}_2\sigma_2\mathcal{A}_2\subset\Aut(\FF_2)$.

A general element of $\psi_2\varphi_2(\Delta)$ does not intersect $E$ and intersects $C$ in exactly one point different from $p_2$ (Remark~\ref{def inverse}~\ref{rmk pt},\ref{rmk inters}). Therefore, $[1:0:0]$, the point $p$ infinitely near to it corresponding to the tangent direction $\{y=0\}$ are  base-points of the linear system $(\varphi_2)^{-1}\psi_2\varphi_2(\Delta)$. The third base-point correspond via $\varphi_2$ to the point $\psi_2(p_2)$. In particular, it is infinitely near to $p$ if and only if $\psi_2(p_2)\in C$ and it is a proper point of $\PP^2$ otherwise. It follows from Remark~\ref{rmk linsyst} that we can write $(\varphi_2)^{-1}\psi_2\varphi_2=\beta\sigma_j\alpha$ for $j=1,2$ and $\alpha,\beta\in\Aut(\PP^2)$ where $\alpha,\beta^{-1}$ respectively send the linear system of $(\varphi_2)^{-1}\psi_2\varphi_2$, $(\varphi_2)^{-1}(\psi_2)^{-1}\varphi_2$ onto the linear system of $\sigma_j$. It follows that $\alpha,\beta$ fix $[1:0:0],p$ and hence $\alpha,\beta\in\mathcal{A}_2$. 

Note that $(\varphi_2)^{-1}\psi_2\varphi_2$ has exactly one proper base-points in $\PP^2$ if and only if $\psi_2$ sends $p_2$ into $C\setminus\{p_2\}$. 
 \end{proof}
 
Lemma~\ref{lem auto linear} allows us to present the following classical results and also describe a Zariski-open set in each $\Aut(\FF_i)$ which will be useful in Section~\ref{5} when we prove that $\Bir(\PP^2)$ is compactly presented.
 
\begin{Lem}\label{cor auto linear}
 For $i=0,2$, let $\mathcal{A}_i:=\Aut(\PP^2)\cap\Aut(\FF_i)$. 
\begin{enumerate}
\item\label{41} The groups $\Aut(\FF_0)$ and $\Aut(\FF_2)$ are linear algebraic subgroups of $\Bir(\PP^2)$.
\item\label{42} The group $\Aut(\FF_0)$ has two irreducible components, namely the component $\Aut(\FF_0)^0$ containing $1$ and $\tau_{12}\Aut(\FF_0)^0$, where $\tau_{12}\in\Aut(\FF_0)$ is given by $\tau_{12}\colon[x:y:z]\mapsto[y:x:z]$. 
\item\label{43} The group $\Aut(\FF_2)$ is irreducible.  
\item\label{44} The set $\mathcal{A}_0\sigma_3\mathcal{A}_0$ is a Zariski-open set of $\Aut(\FF_0)$. 
\item\label{45} The set $\mathcal{A}_2\sigma_2\mathcal{A}_2$ is a Zariski-open set of $\Aut(\FF_2)$.
\end{enumerate}
\end{Lem}

\begin{proof}
\ref{41}, \ref{42} and \ref{43} are classical results, which for example can be found in \cite[Proposition~2.2.6, Th\'eor\`eme~2]{Bla09}.

\ref{44} By Lemma~\ref{lem auto linear} the set $\Aut(\FF_0)\setminus(\mathcal{A}_0\sigma_3\mathcal{A}_0)$ is the set of elements of $\Aut(\FF_0)$ that send the point $p_0$ into the curve $C_1\cup C_2$ and is therefore closed.

\ref{45} By Lemma~\ref{lem auto linear} the set $\Aut(\FF_2)\setminus(\mathcal{A}_2\sigma_2\mathcal{A}_2)$ is the set of element of $\Aut(\FF_2)$ that fix the curve $C$ and is therefore closed. 
\end{proof}

\begin{Rmk}\label{rmk nc} The Noether-Castelnuvo Theorem states that $\Bir(\PP^2)$ is generated by $\Aut(\PP^2)$ and $\sigma_3$ \cite{Cas} (see also \cite[\S8]{Alb02}). Furthermore, we can write $\sigma_3=\tau_{12}\sigma_2\tau_{12}\sigma_2$ where $\tau_{12}([x:y:z])=([y:x:z])$, hence the group $\Bir(\PP^2)$ is also generated by $\Aut(\PP^2)$ and $\sigma_2$. Therefore, for any $i=0,2$, the group $\Bir(\PP^2)$ is generated by its subgroups $\Aut(\PP^2)$ and $\Aut(\FF_i)$ and thus also generated by all three subgroups $\Aut(\PP^2)$, $\Aut(\FF_0)$ and $\Aut(\FF_2)$.

Note that Lemma~\ref{lem auto linear} in particular implies that all elements of $\Aut(\FF_0)\cup\Aut(\FF_2)$ are linear or quadratic. 
\end{Rmk}

\vskip\baselineskip

\begin{Def}
For a set $S$ let $F_S$ be the free group generated by $S$. 

\noindent For the set $S:=\Aut(\PP^2)\cup\Aut(\FF_0)\cup\Aut(\FF_2)\subset\Bir(\PP^2)$, define 
\[\mathfrak{G}:=F_{S}/\left\langle\begin{array}{ll} fgh^{-1},\quad \text{if}\ fg=h\ \text{in}\ \Aut(\PP^2)\\ fgh^{-1},\quad \text{if}\ fg=h\ \text{in}\ \Aut(\FF_0)\\fgh^{-1}, \quad \text{if}\ fg=h\ \text{in}\ \Aut(\FF_2)\\ \tau_{13}\sigma_3\tau_{13}\sigma_3\end{array}\right\rangle\]
where $\tau_{13}\in\Aut(\PP^2)$ is given by $\tau_{13}\colon[x:y:z]\mapsto[z:y:x]$.  
\end{Def}

\vskip\baselineskip

\begin{Rmk}The group $\mathfrak{G}$ is isomorphic to the free product of the three groups $\Aut(\PP^2),\Aut(\FF_2),\Aut(\FF_2)$ amalgamated along all the pairwise intersections (generalised amalgamated product of the three groups) modulo the relation $\tau_{13}\sigma_3=\sigma_3\tau_{13}$.

A geometric approach to generalised amalgamated products can be found in \cite{GHV90}, \cite{Ser77} and \cite{Sta90}. The generalised amalgamated product
$$F_S/ \left\langle\begin{array}{ll} fgh^{-1},\quad \text{if}\ fg=h\ \text{in}\ \Aut(\PP^2)\\ fgh^{-1},\quad \text{if}\ fg=h\ \text{in}\ \Aut(\FF_0)\\fgh^{-1}, \quad \text{if}\ fg=h\ \text{in}\ \Aut(\FF_2)\end{array}\right\rangle$$
is in \cite[\S1.3]{Sta90} the colimit of the diagram 
\[\xymatrix{&&\Aut(\PP^2)\ar@{-}[rd]\ar@{-}[ld]&&\\ &\mathcal{A}_0\ar@{-}[r]\ar@{-}[ld]&\mathcal{A}_0\cap\mathcal{A}_2&\mathcal{A}_2\ar@{-}[l]\ar@{-}[rd]\\ \Aut(\FF_0)\ar@{-}[rr]&&\Aut(\FF_0)\cap\Aut(\FF_2)\ar@{-}[rr]\ar@{-}[u]&&\Aut(\FF_2)}\]

Equivalently, it is the fundamental group of a 2-complex of groups. The vertices are $\Aut(\PP^2)$, $\Aut(\FF_0)$, $\Aut(\FF_2)$, the edges are their pairwise intersection and the 2-simplex is the group $\Aut(\PP^2)\cap\Aut(\FF_0)\cap\Aut(\FF_2)$ \cite[\S2.1, \S3.4]{GHV90}, \cite[4.4]{Ser77}.
\end{Rmk}

\begin{Rmk}\label{rmk w}
By Remark~\ref{rmk nc} there exists a canonical surjective homomorphism of groups 
\[\pi:\mathfrak{G}\rightarrow\Bir(\PP^2)\] 
and by definition of $\mathfrak{G}$ a natural map
\[ w\colon\Aut(\PP^2)\cup\Aut(\FF_0)\cup\Aut(\FF_2)\rightarrow\mathfrak{G}\] 
which sends an element to its corresponding word. Note that $\pi\circ w$ is the identity map.\end{Rmk} 

\section{Base-points, Multiplicities, de Jonquières}\label{2}

The methods we use mainly consist of studying linear systems of $\PP^2$ and their base-points. In this section we recall some definitions, notions and formulae which will be used almost constantly in Section 4 and 5, which have the aim to prove Theorem B (Theorem \ref{amalg}).

\begin{Def}
A {\em point over} $\PP^2$ is a point $p\in S$, where \mbox{$S:=S_{n+1}\stackrel{\nu_n}\rightarrow S_{n-1}\stackrel{\nu_{n-1}}\rightarrow\cdots\stackrel{\nu_1}\rightarrow S_0:=\PP^2$} is a sequence of blow-ups, and where we identify $p\in S$ with $p_i\in S_i$ if $\nu_{i+1}\cdots\nu_n\colon S\rightarrow S_i$ is a local isomorphism around $p$ sending $p$ to $p_i$. 

A point $p\in S$ over $\PP^2$ is {\em proper} if it is equivalent to a point $p'\in\PP^2$, and {\em infinitely near} otherwise. 
\end{Def}

\begin{Def}\label{def consistent}
Let $f\in\Bir(\PP^2)$ be a quadratic birational transformation and call $p_1,p_2,p_3$ its base-points and $q_1,q_2,q_3$ the base-points of $f^{-1}$. We say that {\it the base-points of $f$ are ordered consistently} if the following holds: The base-points of $f$ and of $f^{-1}$ are ordered such that
\begin{enumerate}
\item If $p_1,p_2,p_3$ are proper points of $\PP^2$ then all the lines through $p_1$ (respectively $p_2$, $p_3$) are sent onto lines through $q_1$ (respectively $q_2$, $q_3$). 
\item If $p_1,p_2$ are proper points of $\PP^2$ and $p_3$ is infinitely near to $p_1$, then all the lines through $p_1$ (respectively $p_2$) are sent onto lines through $q_1$ (respectively $q_2$). 
\item If $p_1$ is a proper point of $\PP^2$, $p_2$ infinitely near $p_1$ and $p_3$ infinitely near $p_2$ then the lines through $p_1$ are sent onto lines through $q_1$ and the exceptional curve associated to $p_3$ is sent onto the tangent associated to $q_2$.  
\end{enumerate}
\end{Def}

\begin{Rmk}
Writing down the blow-up diagram of the three quadratic involutions $\sigma_1,\sigma_2,\sigma_3$, we see that we can always order their base-points consistenly (for example for $\sigma_3$ the ordering $p_1=q_1=[1:0:0],p_2=q_2=[0:1:0],p_3=q_3=[0:0:1]$ is consistent). Since any quadratic birational transformation of $\PP^2$ can be written $\beta\sigma_i\alpha$ for some suitable $i\in\{1,2,3\}$, $\alpha,\beta\in\Aut(\PP^2)$ (Remark~\ref{rmk linsyst}), it is always possible to order its base-points consistently. 

Throughout the article, we will always assume that the base-points of a quadratic transformation of $\PP^2$ are ordered consistently.
\end{Rmk}

Let us remind of the following formula: Let $\Delta$ be a linear system and $f\in\Bir(\PP^2)$ a quadratic transformation with base-points $p_1,p_2,p_3$, and $q_1,q_2,q_3$ the base-points of $f^{-1}$. Let $a_i$ be the multiplicity of $\Delta$ in $p_i$ and $b_i$ the multiplicity of $f(\Delta)$ in $q_i$. If the base-points of $f$ are ordered consistently then 
\[\deg(f(\Delta))=2\deg(\Delta)-\varepsilon,\qquad b_i=\deg(\Delta)-\varepsilon+a_i\]
for $i=1,2,3$ and $\varepsilon=a_1+a_2+a_3$ \cite[\S4.2]{Alb02}.

\begin{Def}
We define 
$$J:=\{f\in\Bir(\PP^2):f\ \text{preserves the pencil of lines through}\ [1:0:0]\}.$$ The elements of $J$ are called {\it de Jonquières transformations}. 

A linear system $\Delta$ of $\PP^2$ of degree $\deg(\Delta)=d$ and with base-points $p_1,\dots,p_n$ of multiplicity $a_1,\dots,a_n$ is called {\it de Jonquières linear system} if it has multiplicity $d-1$ at $[1:0:0]$ and satisfies the conditions $d^2-1=\sum_{i=1}^na_i^2$ and  $3(d-1)=\sum_{i=1}^na_i$.

We call a base-point of $f$ a {\it simple base-point} if it is different from $[1:0:0]$ and denote the set of simple base-points by $\Bp(f)$. 
\end{Def}

\begin{Rmk}\label{rmk dJ}\item
(1) We have the following inclusions: $\Aut(\FF_2)\subset J$ and $\Aut(\FF_0)^0\subset J$, where $\Aut(\FF_0)^0$ is the connected component of $\Aut(\FF_0)$ containing $\Id$ and which is equal to $(\mathcal{A}\sigma_2\mathcal{A})\cup(\mathcal{A}\tau_{12}\sigma_2\tau_{12}\mathcal{A})\cup(\mathcal{A}\sigma_3\mathcal{A})\cup\mathcal{A}$, where $\mathcal{A}=\{\alpha\in\Aut(\PP^2)\cap\Aut(\FF_2)\mid\alpha([1:0:0])=[1:0:0]\}$ and $\tau_{12}\in\Aut(\PP^2)$ is given by $\tau_{12}\colon[x:y:z]\mapsto[y:x:z]$ (Lemma \ref{cor auto linear}).
 
(2) Any element of $f\in J\backslash\Aut(\PP^2)$ of degree $d$ has $2d-1$ base-points: the base-point $[1:0:0]$ of multiplicity $d-1$ and $2d-2$ other base-points of multiplicity one (this follows from the conditions on the degree and multiplicities). Thus the definition of simple base-point of $f$ is quite natural. If $f\in J$ is of degree $2$, it has exactly three base-points, all of multiplicity one. Its simple base-points are just the ones different from $[1:0:0]$.

(3) A de Jonquières linear system of $\PP^2$ of degree $d$ has $2d-1$ base-points and the multiplicity at any base-point different from $[1:0:0]$ is one. Such a point is called a {\it simple base-point} of $\Delta$.  Observe that for $f\in J$ and $\Delta$ a de Jonquières linear system, $f(\Delta)$ is a de Jonquières linear system, and the linear system of $f$ is a de Jonquières linear system.
\end{Rmk}

\begin{Lem}\label{lem dJ}
For any quadratic de Jonquières transformation $f\in\Bir(\PP^2)$ there exist $\alpha_1,\alpha_2\in\Aut(\PP^2)\cap J$, $\tau\in\{\sigma_1,\sigma_2,\tau_{12}\sigma_2\tau_{12},\sigma_3\}\subset\Aut(\FF_0)\cup\Aut(\FF_2)$, where $\tau_{12}\in\Aut(\PP^2)$ is given by $\tau_{12}:[x:y:z]\mapsto[y:x:z]$, such that $f=\alpha_2\tau\alpha_1$. 

In particular, $(\Aut(\PP^2)\cup\Aut(\FF_0)\cup\Aut(\FF_2))\cap\mathcal{J}$ generates $\mathcal{J}$. 
\end{Lem}

\begin{proof}
By Remark~\ref{def inverse}, we can write $f=\alpha_2\sigma_i\alpha_1$ for some $\alpha_1,\alpha_2\in\Bir(\PP^2)$ and $i$ is determined by the amount of proper base-points of $f$. Since $f$ is de Jonquières the point $[1:0:0]$ is a base-point of $f$. 

If $f$ has only one proper base-point in $\PP^2$, it has to be fixed by $\alpha_1$ and $\alpha_2$, which belong thus to $J$. This gives the result. 

Suppose that $f$ has exactly two proper base-points, namely $[1:0:0]$ and $p$. This implies that $\sigma_i=\sigma_2$, which has base-points $[1:0:0],[0:1:0]$ and a third one, infinitely near $[1:0:0]$. The base-point of $f$ which is not a proper point of $\PP^2$ is either infinitely near $[1:0:0]$ or $p$. If it is infinitely near $[1:0:0]$, then $\alpha_1,\alpha_2$ fix $[1:0:0]$ and are therefore de Jonquières. If it is infinitely near $p$ then $\alpha_1$ sends $p$ onto $[1:0:0]$ and $[1:0:0]$ onto $[0:1:0]$. We write $\alpha_1=\tau_{12}\beta_1$, $\alpha_2=\beta_2\tau_{12}$, for some $\beta_1,\beta_2\in\Aut(\PP^2)$, which means that $\beta_1$ fixes $[1:0:0]$, i.e. $\beta_1\in J\cap\Aut(\PP^2)$ and $f=\beta_2(\tau_{12}\sigma_2\tau_{12})\beta_1$. Since $f,\beta_1,\tau_{12}\sigma_2\tau_{12}\in J$, we have $\beta_2\in J$.

Suppose that $f$ has three proper base-points. For any $\theta\in\Aut(\PP^2)$ that permutes the coordinate $x,y,z$, we have $\theta\sigma_3\theta=\sigma_3$. Therefore, we can assume that $\alpha_1\in J$. Since $\sigma_3$ is de Jonqui\`eres, the map $\alpha_2$ has to be de Jonqui\`eres as well.

Since every element of $\mathcal{J}$ decomposes into quadratic elements of $\mathcal{J}$ \cite[Theorem 8.4.3]{Alb02}, Lemma~\ref{lem dJ} implies that $\mathcal{J}$ is generated by $(\Aut(\PP^2)\cup\Aut(\FF_0)\cup\Aut(\FF_2))\cap\mathcal{J}$ generates $\mathcal{J}$.
\end{proof}

\begin{Rmk}\label{deg dJ}
Suppose $\Delta$ is a de Jonquières linear system of degree $d$ and $f$ a quadratic de Jonquières transformation. We can say the following about the degree of $f(\Delta)$: Let $p_1=[1:0:0],p_2,p_3$ be the base-points of $f$ and $a_i$ be the multiplicity of $\Delta$ in $p_i$. Then $a_1=\deg(\Delta)-1$ and by the formula given above, we have
\[\deg(f(\Delta))=2d-(d-1)-a_2-a_3=d+1-a_2-a_3\]
Since $\Delta$ has one base-point of multiplicity $d-1$ and all the other base-points are of multiplicity one, we know that for $i=2,3$, $a_i$ is either zero or one. In fact, $a_i=0$ if $p_i$ is not a common base-point of $f$ and $\Delta$, and $a_i=1$ if $p_i$ is a common base-point of $f$ and $\Delta$. Thus the formula implies
\[\deg(f(\Delta))=\begin{cases} d+1,&\text{if}\ f\ \text{and}\ \Delta\ \text{have no common simple base-points}\\ d,&\text{if}\ f\ \text{and}\ \Delta\ \text{have exactly one common simple base-point}\\ d-1,&\text{if}\ f\ \text{and}\ \Delta\ \text{have exactly two common simple base-points}\end{cases}\]
Furthermore, if $p$ is a simple base-point of $\Delta$ that is not base-point of $f$, then $f^{\bullet}(p)$ (see definition below) is a simple base-point of $f(\Delta)$ \cite[\S4.1]{Alb02}.
\end{Rmk}

\begin{Def}
Let $f\in\Bir(\PP^2)$ and $p$ a point over the domain $\PP^2$ that is not a base-point of $f$. 
Take a minimal resolution of $f$
\[ \xymatrix{&S\ar[rd]^{\nu_2}\ar[dl]_{\nu_1}&\\ \PP^2\ar@{-->}[rr]^{f}&&\PP^2} \]
where $\nu_1,\nu_2$ are sequences of blow-ups. Let $p'\in S$ be a representative of $p$. We can see $p'$ as a point over the range $\PP^2$, and call it $f^{\bullet}(p)$.
\end{Def}

Lets look at an example to understand $f^{\bullet}(p)$ and $f(p)$:
\begin{Ex} 
Consider the standard quadratic involution $\sigma_3\in\Bir(\PP^2)$ and the point $p=[0:1:1]$, which is on the line $\{x=0\}$ contracted by $\sigma_3$ onto the point $[1:0:0]$, which means that $\sigma_3(p)=[1:0:0]$. 
The line $L=\{y=z\}$ passing through $p$ and $[1:0:0]$ is sent by $\sigma_3$ onto itself. By definition, $(\sigma_3)^{\bullet}(p)$ is the point in the first neighbourhood of $[1:0:0]$ corresponding to the tangent direction $\{y=z\}$. In conclusion, $\sigma_3(p)$ is a proper point of $\PP^2$, whereas $(\sigma_3)^{\bullet}(p)$ is not. The following picture (Figure~\ref{point}) illustrates the situation.
\end{Ex}

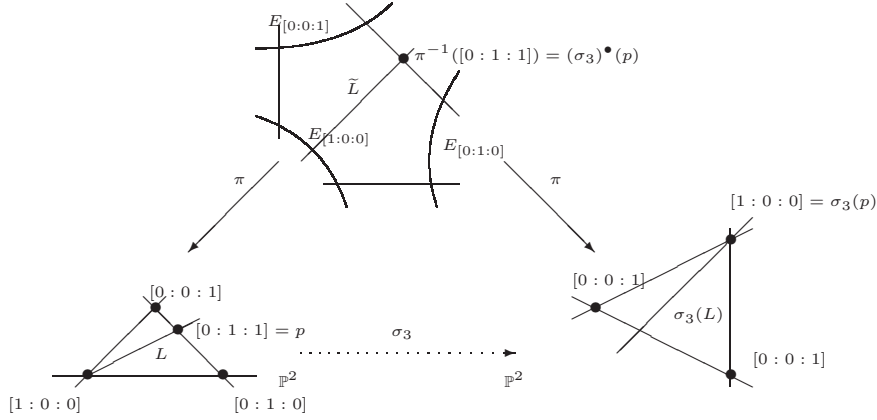
\begin{figure}[H]
\begin{center}
\begin{picture}(125,135)
\setlength{\unitlength}{3mm}
\put(-10,0.5){\line (1,0){9}}
\put(-8.75,0.25){\unboldmath$\bullet$}
\put(-12,-1){\tiny$[1:0:0]$}
\put(-2.75,0.25){\unboldmath$\bullet$}
\put(-2,-1){\tiny$[0:1:0]$}
\multiput(0,0)(10,0){2}{\tiny$\PP^2$}
\put(-9,0){\line (1,1){4}}
\put(-2,0){\line (-1,1){4}}
\put(-5.75,3.25){\unboldmath$\bullet$}
\put(-5.75,4){\tiny$[0:0:1]$}
\put(-4.75,2.25){\unboldmath$\bullet$}
\put(-3.7,2.3){\tiny$[0:1:1]=p$}
\put(-8.5,0.5){\line(2,1){5}}
\put(-5.5,1.25){\tiny$L$}
\qbezier[20](1,1.5)(5,1.5)(10,1.5)
\put(10,1.5){\vector(1,0){0.5}}
\put(5,2.25){\tiny$\sigma_3$}
\put(0,10){\vector(-1,-1){4}}
\put(10,10){\vector(1,-1){4}}
\multiput(-2,9)(14,0){2}{\tiny$\pi$}
\put(2,9){\line(1,0){6}}
\put(0,11){\line(0,1){5}}
\put(3,17){\line(1,-1){5}}
\qbezier(3,8)(2,11)(-1,12)
\qbezier(-1,15)(3,15)(5,17)
\qbezier(7,8)(6,11)(8,14)
\put(1.25,11){\tiny$E_{[1:0:0]}$}
\put(-0.5,16){\tiny$E_{[0:0:1]}$}
\put(7.25,10.5){\tiny$E_{[0:1:0]}$}
\put(5.25,14.25){\unboldmath$\bullet$}
\put(6,14.5){\tiny$\pi^{-1}([0:1:1])=(\sigma_3)^{\bullet}(p)$}
\put(6,15){\line(-1,-1){5}}
\put(3,13){\tiny$\widetilde{L}$}
\put(13,4){\line(2,-1){8}}
\put(13,3){\line(2,1){8}}
\put(20,7){\line(0,-1){7}}
\put(21,7.5){\line(-1,-1){6}}
\put(13.75,3.25){\unboldmath$\bullet$}
\put(19.75,0.25){\unboldmath$\bullet$}
\put(19.75,6.25){\unboldmath$\bullet$}
\put(13,4.5){\tiny$[0:0:1]$}
\put(21,1){\tiny$[0:0:1]$}
\put(20,8){\tiny$[1:0:0]=\sigma_3(p)$}
\put(17.5,3){\tiny$\sigma_3(L)$}
\end{picture}
\end{center}
\caption{The points $(\sigma_3)^{\bullet}(p)$ and $\sigma_3(p)$}\label{point}
\end{figure}

\begin{Rmk}
Note that $f^{\bullet}$ is a one-to-one correspondence between the sets
\begin{align*}&(\PP^2\cup\{\text{infinitely near points}\})\setminus\{\text{base-points of}\ f\}\quad\text{and}\\
&(\PP^2\cup\{\text{infinitely near points}\})\setminus\{\text{base-points of}\ f^{-1}\}\end{align*}
\end{Rmk}

\section{Basic relations in $\mathfrak{G}$}\label{3}

In this section, we present basic relations that hold in $\mathfrak{G}$ and which will be the backbone of the proof of Theorem A (Theorem~\ref{amalg}). We prove relations for words in $\mathfrak{G}$ of length three using properties of the elements of $\Aut(\FF_0)$, $\Aut(\FF_2)$ and $\Aut(\PP^2)$. 
(Lemma~\ref{lem degree 1,2}, Lemma~\ref{lem square dJ} and Lemma~\ref{lem deg 3 proper base-point}). They will then be used in the next section to prove that there exists an injective map $w_J\colon J\rightarrow\mathfrak{G}$ such that $\pi\circ w_J=\Id$ (Lemma~\ref{lem dJ Id}, Corollary~\ref{cor dJ Id}) which will enable us to prove Theorem A (Theorem~\ref{amalg}) using the result that $\Bir(\PP^2)$ is the amalgamated product of $\Aut(\PP^2)$ and $J$ modulo one relation \cite{Bla12} (Theorem~\ref{bla}).

Lemma~\ref{lem permutations} and \ref{lem degree 1,2} yield that words of length three in $\mathfrak{G}$ whose image in $\Bir(\PP^2)$ is linear or quadratic behave like their images in $\Bir(\PP^2)$. Lemma~\ref{lem square dJ} and \ref{lem deg 3 proper base-point} yield relations for words of length three whose image in $\Bir(\PP^2)$ is de Jonquières and of degree three.

\vskip\baselineskip

Define $\TAut(\PP^2)=D\rtimes S_3$, where $S_3\subset\Aut(\PP^2)$ is the image of the permutation matrices of $\GL_3$ and $D$ is the image of the three dimensional torus. We can check that the group $\TAut(\PP^2)$ is normalised by $\sigma_3$, and the automorphism of $\TAut(\PP^2)$ given by the conjugation of $\sigma_3$ will be denoted by $\iota$. Note that $\iota(\alpha)=\alpha$ for $\alpha\in S_3$ and $\iota(\delta)=\delta^{-1}$ for $\delta\in D$.

As subgroup of $\Aut(\PP^2)$, we can embed $\TAut(\PP^2)$ (as a set) into $\mathfrak{G}$ by the word map $w$. The next lemma shows that in $\mathfrak{G}$ the image of $\TAut(\PP^2)$ is normalised by $w(\sigma_3)$:

\begin{Lem}\label{lem permutations}
For any $(\delta,\alpha)\in D\rtimes S_3$ the relation $w(\delta\alpha)w( \sigma_3)=w(\sigma_3)w(\iota(\delta\alpha))$ holds in $\mathfrak{G}$. 
\end{Lem}

\begin{proof}
Let $\tau_{12}:[x:y:z]\mapsto[y:x:z]$. In $\Aut(\FF_0)$, the relation $\tau_{12}\sigma_3\tau_{12}=\sigma_3$ holds, hence the relation $w(\tau_{12})w(\sigma_3)=w(\sigma_3)w(\tau_{12})$ holds in $\mathfrak{G}$. By definition, $w(\tau_{13})w(\sigma_3)=w(\sigma_3)w(\tau_{13})$ is a relation in $\mathfrak{G}$, and $\tau_{13}$ and $\tau_{12}$ generate $S_3$. Therefore, the relation $w(\alpha)w(\sigma_3)=w(\sigma_3)w(\alpha)=w(\sigma_3)w(\iota(\alpha))$ holds in $\mathfrak{G}$ for any $\alpha\in S_3$. 

Let $\delta\in D$. The relation $\delta \sigma_3\delta= \sigma_3$ holds in $\Aut(\FF_0)$, hence $w(\delta)w(\sigma_3)=w(\sigma_3)w(\delta^{-1})=w(\sigma_3)w(\iota(\delta))$ holds in $\mathfrak{G}$. 
\end{proof}

Using Lemma~\ref{lem permutations}, we now show that for $f,g,h\in\Aut(\PP^2)\cup\Aut(\FF_0)\cup\Aut(\FF_2)$ and $\deg(fgh)\leq 2$, the word $w(f)w(g)w(h)$ behaves like the composition $fgh$:

\begin{Lem}\label{lem degree 1,2}
Let $g\in\Aut(\PP^2),h,f\in\Aut(\FF_0)\cup\Aut(\FF_2)$ such that 
$\deg(fgh)\in\{1,2\}$.
\begin{enumerate}
\item If $\deg(fgh)=1$, then $w(f)w(g)w(h)=w(fgh)$ in $\mathfrak{G}$.
\item If $\deg(fgh)=2$, there exist $\alpha,\beta\in\Aut(\PP^2)$, $\tilde{g}\in\Aut(\FF_0)\cup\Aut(\FF_2)$ such that $w(f)w(g)w(h)=w(\beta)w(\tilde{g})w(\alpha)$ in $\mathfrak{G}$.
\item If $\deg(fgh)=2$ and $f,g,h\in J$ then $\alpha,\beta,\tilde{g}$ can be chosen to be in $J$. 
\end{enumerate}
\end{Lem}

\begin{proof}
Suppose that $f\in\Aut(\PP^2)$ or $h\in\Aut(\PP^2)$. The first claim follows from the definition of $\mathfrak{G}$. The second and third claim follow by putting $\beta:=fg$ or $\alpha=gh$ if $\deg(f)=1$ or $\deg(h)=1$ respectively. 

Assume that $h\in\Aut(\FF_i)\backslash\Aut(\PP^2)$ and $f\in\Aut(\FF_j)\backslash\Aut(\PP^2)$. Since $\Aut(\FF_0)$ is generated by $\Aut(\PP^2)\cap\Aut(\FF_0)$ and $\sigma_2,\sigma_3$ and $\Aut(\FF_2)$ is generated by $\Aut(\PP^2)\cap\Aut(\FF_2)$ and $\sigma_1,\sigma_2$ (Lemma~\ref{lem auto linear}), we can write $f=\beta_2\sigma_k\alpha_2$ and $h=\beta_1\sigma_l\alpha_1$ for some $\alpha_1,\beta_1\in\Aut(\FF_i)\cap\Aut(\PP^2)$, $\alpha_2,\beta_2\in\Aut(\FF_j)\Aut(\PP^2)$ and $k,l\in\{1,2,3\}$. By replacing $g$ with \mbox{$\alpha_2g\beta_1$} in $\Aut(\PP^2)$, we can assume that $\alpha_2=\beta_1=\Id$ and hence $f=\beta_2\sigma_k$ and $h=\sigma_l\alpha_1$. It follows from Remark~\ref{rmk linsyst} that
\begin{align*}&\text{the base-points of}\ f\ \text{are exactly the base-points of}\ \sigma_k,\tag{$\ast$}\label{Bp}\\ 
&\text{the base-points of}\ h^{-1}\ \text{are exactly the base-points of}\ \sigma_l.\end{align*}

(i) Suppose that $\deg(fgh)=1$. Then $f$ and $(gh)^{-1}$ have exactly the same base-points, which are respectively the base-points of $\sigma_k$ and the image of the base-points of $\sigma_l$ by $g$. In particular, $f$, $(gh)^{-1}$ and hence also $\sigma_k$, $\sigma_l$ have the same amount of proper base-points in $\PP^2$. Since $\sigma_1$, $\sigma_2$, $\sigma_3$ have exactly one, two and three proper proper base-points, it follows that $\sigma_k=\sigma_l$.

If $k\in\{1,2\}$ the equation $\sigma_k=\sigma_l$, the fact that $f$, $(gh)^{-1}$ have the same base-points and (\ref{Bp}) imply that $g\in\Aut(\FF_2)\cap\Aut(\PP^2)$ and so $f,g,h\in\Aut(\FF_2)$. The definition of $\mathfrak{G}$ then implies \mbox{$w(f)w(g)w(h)=w(fgh)$}.

If $k=3$, the equation $\sigma_k=\sigma_l$, the fact that $f$, $(gh)^{-1}$ have the same base-points and (\ref{Bp}) imply that $g$ permutes the base-points of $\sigma_3$. Lemma~\ref{lem permutations} states that \mbox{$w(g)w(\sigma_3)=w(\sigma_3)w(\iota(g))$}. We get
\begin{align*}w(f)w(g)w(h)&=w(\beta_2\sigma_3)w(g)w(\sigma_3\alpha_1)\\&=w(\beta_2)w(\sigma_3)w(\sigma_3)w(\iota(g))w(\alpha_1)\\&=w(\beta_2)w(\iota(g))w(\alpha_1)=w(\beta_2\iota(g)\alpha_1)\\&=w(\beta_2\sigma_3g\sigma_3\alpha_1)=w(fgh).\end{align*}

(ii) Suppose $\deg(fgh)=2$, i.e. $f$ and $(gh)^{-1}$ have exactly two common base-points $s,t$, at least one of them being proper. Assume that $s$ is proper. 

If $t$ is infinitely near to $s$, (\ref{Bp}) implies that \mbox{$\{k,l\}\subset\{1,2\}$}, i.e. $f,h\in\Aut(\FF_2)$. Then (\ref{Bp}) and the fact that $t$ is infinitely near $s$ implies that $s=[1:0:0]$ and that $t$ lies on the strict transform of \mbox{$\{y=0\}$}. Then $s,t$ are base-points of both $h^{-1}$ and $(gh)^{-1}$ and it follows that \mbox{$g(\{s,t\})=\{s,t\}$}, thus $g\in\Aut(\FF_2)$. It follows that in $\Aut(\FF_2)$ (hence also in $\mathfrak{G}$) 
$$w(f)w(g)w(h)=w(\beta_2\sigma_k)w(g)w(\sigma_l\alpha_1)=w(\beta_2)w(\sigma_kg\sigma_l)w(\alpha_1).$$ 
Remark that any map contained in $\Aut(\FF_2)$ is de Jonquières (Remark~\ref{rmk dJ}), from which claim (iii) follows for this subcase. 

If $s$ and $t$ are both proper, (\ref{Bp}) implies that \mbox{$\{k,l\}\subset\{2,3\}$}, i.e. $f,h\in\Aut(\FF_0)$. Then (\ref{Bp}) yields that \mbox{$\{s,t\}\subset\{[1:0:0],[0:1:0],[0:0:1]\}$}. There exist $\alpha,\beta\in \TAut(\PP^2)$ such that 
\begin{align*}&\alpha(\{[1:0:0],[0:1:0]\})=g^{-1}(\{s,t\}),\quad \beta(\{s,t\})=\{[1:0:0],[0:1:0]\}\\ &\beta g\alpha([1:0:0])=[1:0:0],\quad\beta g\alpha([0:1:0])=[0:1:0].\end{align*}  
If $k=2$, we may choose $\beta=\Id$. If $l=2$, we may choose $\alpha=\Id$. We get 
\begin{align*}w(f)w(g)w(h)&=w(\beta_2 \sigma_k)w(\beta^{-1})w(\beta)w(g)w(\alpha)w(\alpha^{-1})w(\sigma_l\alpha_1)\\&\stackrel{\tiny{\text{Lem}\ref{lem permutations}}}=w(\beta_2\iota(\beta^{-1}))w(\sigma_k)w(\beta g\alpha)w(\sigma_l)w(\iota(\alpha^{-1})\alpha_1)\\&=w(\beta_2\iota(\beta^{-1}))w(\sigma_k\beta g\alpha\sigma_l)w(\iota(\alpha)\alpha_1)\end{align*} 
The claim follows with \mbox{$\alpha=\iota(\alpha^{-1})\alpha_1$}, \mbox{$\tilde{g}=\sigma_k(\tilde{\beta}g\tilde{\alpha})\sigma_l$}, \mbox{$\beta=\beta_1\iota(\beta^{-1})$}. 
It remains to prove claim (iii) for this sub case: If $f,g,h$ are de Jonquières, then \mbox{$g([1:0:0])=[1:0:0]$} and $[1:0:0]$ is a common base-point of $gf$ and $h^{-1}$. Choosing $\alpha,\beta$ above such that they fix $[1:0:0]$ (i.e. are de Jonquières) it follows that \mbox{$\tilde{g}=\beta g\alpha$} is de Jonquières. The maps $f$ and $h$ being de Jonquières implies that $\alpha_1,\beta_2$ are de Jonquières (Remark~\ref{rmk dJ}), hence $\iota(\alpha^{-1})\alpha_1$, $\beta_1\iota(\beta^{-1})$ and $\tilde{g}$ are de Jonquières. 
\end{proof}

The next two lemmata yield relations for words of length three whose image in $\Bir(\PP^2)$ is of degree three.

\begin{Lem}\label{lem square dJ}
Let $f\in(\Aut(\FF_0)\cup\Aut(\FF_2))\cap J$ be a quadratic transformation, $\alpha_1,\dots,\alpha_4\in \Aut(\PP^2)\cap J$ such that 
\begin{enumerate} 
\item $f$ is a local isomorphism at the simple base-points $q_2,q_3$ of $(\alpha_2\sigma_3\alpha_1)^{-1}$, 
\item $\Bp(\alpha_4\sigma_3\alpha_3)=\{f(q_2),f(q_3)\}$.
\end{enumerate}
Then
\begin{enumerate} 
\item The map $(\alpha_4\sigma_3\alpha_3)f(\alpha_2\sigma_3\alpha_1)$ is quadratic de Jonquières, 
\item $\Bp((\alpha_4\sigma_3\alpha_3)f(\alpha_2\sigma_3\alpha_1))=\left((\alpha_2\sigma_3\alpha_1)^{-1}\right)^{\bullet}(\Bp(f))$
\item there exist $\beta_1,\beta_3\in\Aut(\PP^2)\cap J$ and $\beta_2\in\{\sigma_2,\sigma_3,\tau_{12}\sigma_2\tau_{12}\}$ such that the following equation holds in $\mathfrak{G}$:
\[w(\alpha_4)w(\sigma_3)w(\alpha_3)w(f)w(\alpha_2)w(\sigma_3)w(\alpha_1)=w(\beta_3)w(\beta_2)w(\beta_1)\]
i.e. the following diagram corresponds to a relation in $\mathfrak{G}$:
\[\xymatrix{ \PP^2\ar@{-->}[rr]^{f}&&\PP^2\ar@{-->}[d]^{\alpha_4\sigma_3\alpha_3}\\ \PP^2\ar@{-->}[u]^{\alpha_2\sigma_3\alpha_1}\ar@{-->}[rr]_{\beta_3\beta_2\beta_1}&&\PP^2 }\]
\end{enumerate}
\end{Lem}

\begin{proof}
Define $\tau_1:=\alpha_2\sigma_3\alpha_1$ and $\tau_2:=\alpha_4\sigma_3\alpha_2$ and denote by $p_1=[1:0:0],p_2,p_3$ the base-points of $f$ and by $p_1,\bar{p}_2,\bar{p}_3$ the base-points of its inverse (ordered consistently, see Definition~\ref{def consistent}). 

Since $f$ is a local isomorphism at $q_2,q_3$ the map $f^{-1}$ is a local isomorphism at $f(q_2),f(q_3)$. Hence there exist simple base-points $p_i,\bar{p}_i$ of $f$,$f^{-1}$ respectively, either proper points of $\PP^2$ or infinitely near $p_1$, which do not lie on the lines contracted by $(\tau_1)^{-1}$ and $\tau_2$. Up to order, we can assume that $p_i=p_2$. Therefore, the points $\tilde{p}_2:=(\tau_1^{-1})^{\bullet}(p_2)$ and $\hat{p}_2:=(\tau_2)^{\bullet}(\bar{p}_2)$ are proper points of $\PP^2$.  

Observe that the map $\tau_2f\tau_1$ is de Jonquières of degree two having base-points $p_1,\tilde{p}_2,\tilde{p}_3:=(\tau_1^{-1})^{\bullet}(p_3)$ and its inverse having base-points $p_1,\hat{p}_2,\hat{p}_3:=(\tau_2)^{\bullet}(\bar{p}_3)$. Indeed, the map $f\tau_1$ is of degree three with base-points \mbox{$p_1,\tilde{p}_2,\tilde{p}_3,\tilde{q}_2,\tilde{q}_3$}, where $\tilde{q}_2,\tilde{q}_3$ are the simple base-points of $\tau_1$, and its inverse having base-points \mbox{$p_1,p_4,p_5,f(q_2),f(q_3)$}. Thus $\tau_2f\tau_1$ is de Jonquières of degree two with base-points $p_1,\tilde{p}_2,\tilde{p}_3$ and its inverse having base-points $p_1,\hat{p}_2,\hat{p}_3$ (by the formula given in Section 3).

Since $\tau_2f\tau_1$ has at least one simple proper base-point (namely $\tilde{p}_2$), Lemma~\ref{lem dJ} and \ref{rmk linsyst} imply that there exist $\beta_1,\beta_2\in\Aut(\PP^2)\cap J$ and $\beta_2\in\{\sigma_2,\sigma_3,\tau_{12}\sigma_2\tau_{12}\}$ such that $\tau_2f\tau_1=\beta_3\beta_2\beta_1$. 

It is left to prove that \mbox{$w(\alpha_4)w(\sigma_3)w(\alpha_3)w(f)w(\alpha_2)w(\sigma_3)w(\alpha_1)=w(\beta_3)w(\beta_2)w(\beta_1)$} in $\mathfrak{G}$. We will use Lemma~\ref{lem degree 1,2}, and for this we fill the diagram
\[\xymatrix{ \PP^2\ar@{-->}[rr]^{f}&&\PP^2\ar@{-->}[d]^{\tau_2}\\ \PP^2\ar@{-->}[u]^{\tau_1}\ar@{-->}[rr]_{\beta_3\beta_2\beta_1}&&\PP^2 }\]
with triangles corresponding to relations in $\mathfrak{G}$. 

The map $f$ is a local isomorphism at $q_2,q_3$ hence the three points $p_1,p_2,q_2$ are not collinear. Since moreover $p_1,q_2$ are both proper points of $\PP^2$ there exists a quadratic map $\rho\in\Bir(\PP^2)\cap J$ which has base-points $p_1,q_2,p_2$. The maps $\rho\tau_1$ and $\rho f^{-1}$ are quadratic de Jonquières maps with base-points $p_1,\tilde{q}_2,\tilde{p}_2$ and $p_1,\bar{p}_2,f(q_2)$ respectively. It follows that also the map $\rho\tau_1(\beta_3\beta_2\beta_1)^{-1}$ is quadratic. The situation is summarised in the following diagram, where all the arrows are quadratic maps and the points in the brackets are the simple base-points of the corresponding quadratic map:

\[\xymatrix{ \PP^2\ar@{-->}[rrrr]^f^(.2){[p_2,p_3]}^(.8){[\bar{p}_2,\bar{p}_3]}\ar@{-->}[rrd]_{\rho}^(.4){[p_2,q_2]} &&&& \PP^2 \ar@{-->}[dll]_(.4){[\bar{p}_2,f(q_2)]} \ar@{-->}[dd]^{\tau_2}^(.2){[f(q_2),f(q_3)]} \\
&&\PP^2&&\\
\PP^2 \ar@{-->}[uu]^{\tau_1}^(.8){[q_2,q_3]}^(.2){[\tilde{q}_2,\tilde{q}_3]}\ar@{-->}[urr]^(.35){[\tilde{q}_2,\tilde{p}_2]} \ar@{-->}[rrrr]^{\beta_3\beta_2\beta_1}_(.2){[\tilde{p}_2,\tilde{p}_3]}_(.8){[\hat{p}_2,\hat{p}_3]} &&&& \PP^2 \ar@{-->}[ull] } \]
Writing $\rho=\gamma_3\gamma_2\gamma_1$ for some $\gamma_1,\gamma_2,\gamma_3\in(\Aut(\PP^2)\cup\Aut(\FF_0)\cup\Aut(\FF_2))\cap J$, only $\gamma_2$ quadratic (possible by Lemma~\ref{lem dJ}), Lemma~\ref{lem degree 1,2} implies that each triangle in the above diagram corresponds to a relation in $\mathfrak{G}$, making the whole diagram correspond to a relation in $\mathfrak{G}$. 
\end{proof}

\vskip\baselineskip

\begin{Lem}\label{lem deg 3 proper base-point}
Let $f,h\in\Aut(\FF_0)\cup\Aut(\FF_2)$, $g\in\Aut(\PP^2)$, $f,g,h\in J$, and let $\Delta$ be a de Jonquières linear system. Assume that 
$$\deg(fgh)=3,\quad \deg(fgh(\Delta))<\deg(gh(\Delta)),\quad\deg(\Delta)\leq\deg(gh(\Delta))$$ 
and that $(gh)(\Delta)$ has a proper base-point different from $[1:0:0]$. Then there exist $\alpha_1,\dots,\alpha_7\in\left(\Aut(\PP^2)\cup\Aut(\FF_0)\cup\Aut(\FF_2)\right)\cap J$, $\alpha_1,\alpha_3,\alpha_5,\alpha_7\in\Aut(\PP^2)$, such that 
\begin{enumerate}
\item the following equation holds in $\mathfrak{G}$:  
$$w(f)w(g)w(h)=w(\alpha_7)\cdots w(\alpha_1)$$
i.e. the following diagram corresponds to a relation in $\mathfrak{G}$:
\[\xymatrix{&&(gh)(\Delta)\ar@{-->}[rrd]^{f}\ar@{-->}[lld]_{(gh)^{-1}}&& \\
\Delta\ar@{-->}[rd]^{\alpha_2\alpha_1}&&&& fgh(\Delta)\\
&\alpha_1\alpha_2(\Delta)\ar@{-->}[rr]^{\alpha_5\alpha_4\alpha_3}&&\alpha_5\cdots\alpha_1(\Delta)\ar@{-->}[ur]_{\alpha_7\alpha_6}&}\]
\item For $i=2,\dots,7$
$$\deg(\alpha_i\cdots\alpha_1(\Delta))<\deg((gh)(\Delta)).$$
\end{enumerate}
\end{Lem}

\begin{proof}
The equality $\deg(fgh)=3$ implies that $f$ and $(gh)^{-1}$ have exactly one common base-point, namely $p_1=[1:0:0]$. Denote \mbox{$\Bp((gh)^{-1})=\{p_2,p_3\}$}, \mbox{$\Bp(f)=\{p_4,p_5\}$} the simple base-points of $(gh)^{-1}$ and $f$ and write $d=\deg(gh(\Delta))$. 

By assumption, $gh(\Delta)$ is a de Jonquières linear system which has a proper base-point $s$ different from $p_1$. For any point $r$, let $m(r)$ be the multiplicity of $gh(\Delta)$ in $r$ respectively. Then \mbox{$m(p_1)=d-1$} and \mbox{$m(s)=1$} (Remark~\ref{rmk dJ}), and Remark~\ref{deg dJ} implies that because $\deg(\Delta)\leq d$ and $\deg(fgh(\Delta))<d$, we have (up to ordering of $p_2,p_3$)
\begin{align*}&m(p_2)=1,\ m(p_3)\leq 1\\&\deg(fgh(\Delta))=d-1,\quad m(p_4)=m(p_5)=1\end{align*}

We will now construct $\alpha_1,\dots,\alpha_7$.

Assume that $s\in\{p_2,p_3,p_4,p_5\}$. If \mbox{$s\in\{p_2,p_3\}$}, we choose \mbox{$r\in\{p_4,p_5\}$}. If \mbox{$s\in\{p_4,p_5\}$}, we choose \mbox{$r\in\{p_i:i=2,3\ \text{and}\ m(p_i)=1\}$.} We choose $r$ to be infinitely near $p_1$ or a proper point (this is always possible). The points $p_1,s,r$ are not aligned, because \mbox{$a_1+m(s)+m(r)>d$}, thus there exist $\rho\in\Bir(\PP^2)$ quadratic de Jonquières with base-points \mbox{$p_1,r,s$}. The following commutative diagram, where the points in the brackets are the base-points of the corresponding map, summarises the situation:
\[\xymatrix{&&gh(\Delta)\ar@{-->}[d]^{\rho}_{[p_1,s,r]}\ar@{-->}[drr]^{f}^(.25){[p_1,p_4,p_5]}\ar@{-->}[dll]_(.6){(gh)^{-1}}_(.3){[p_1,p_2,p_3]}&&\\ 
\Delta\ar@{-->}[rr]&&\rho gh(\Delta)\ar@{-->}[rr]&&fgh(\Delta)}\]
Using Remark~\ref{deg dJ}, we obtain
\begin{align*}&\deg(\rho gh)=\deg(f\rho^{-1})=2,\\ 
&\deg(\rho gh(\Delta))=d-1<\deg(gh(\Delta))\end{align*} 
We write $\rho=\gamma\tilde{\rho}\delta$, $\rho gh=\alpha_3\alpha_2\alpha_1$, $f\rho^{-1}=\alpha_6\alpha_5\alpha_4$ where $\delta,\gamma,\alpha_1,\cdots,\alpha_6\in\left(\Aut(\PP^2)\cup\Aut(\FF_0)\cup\Aut(\FF_2)\right)\cap J$, only $\tilde{\rho},\alpha_2,\alpha_5$ quadratic (Lemma~\ref{lem dJ}). By Lemma~\ref{lem degree 1,2}, the above diagram is generated by relations in $\mathfrak{G}$. Hence $w(f)w(g)w(h)=w(\alpha_6)\cdots w(\alpha_1)$ in $\mathfrak{G}$.

Assume \mbox{$s\notin\{p_2,p_3,p_4,p_5\}$}, we choose \mbox{$r_1\in\{p_i:i=2,3\ \text{and}\ m(p_i)=1\}$}, \mbox{$r_2\in\{p_4,p_5\}$} such that $r_1$ (respectively $r_2$) is either a proper point or infinitely near $p_1$ (this is always possible). For $i=1,2$, the points $p_1,s,r_i$ are not collinear, because \mbox{$a_1+m(s)+m(r_i)>d$}. Thus there exist $\rho_1,\rho_2\in\Bir(\PP^2)$ quadratic de Jonquières with base-points $p_1,s,r_1$ and $p_1,s,r_2$ respectively. The following commutative diagram, where the brackets are the base-points of the corresponding map, summarises the situation:
\[\xymatrix{&&gh(\Delta)\ar@{-->}[lldd]_{(gh)^{-1}}_(.2){[p_1,p_2,p_3]}\ar@{-->}[ddl]^{\rho_1}|(.4){[p_1,s,r_1]}\ar@{-->}[ddr]_{\rho_2}|(.4){[p_1,s,r_2]}\ar@{-->}[ddrr]^{f}^(.2){[p_1,p_4,p_5]}&&\\
&&&&\\
\Delta\ar@{-->}[r]&\rho_1gh(\Delta)\ar@{-->}[rr]&&\rho_2gh(\Delta)\ar@{-->}[r]&fgh(\Delta)}\]
Using Remark~\ref{deg dJ} we obtain
\begin{align*}&\deg(\rho_1gh)=\deg(\rho_2\rho_1^{-1})=\deg(f\rho_2^{-1})=2,\\ 
&\deg(\rho_1gh(\Delta))=d-1<\deg(gh(\Delta))\\ &\deg(\rho_2gh(\Delta))=d-1<\deg(gh(\Delta))\end{align*}
We write $\rho_1=\gamma_1\tilde{\rho}_1\beta_1$, $\rho_2=\gamma_2\tilde{\rho_2}\beta_2$, $\rho_1gh=\alpha_3\alpha_2\alpha_1$, $\rho_2\rho_1^{-1}=\alpha_6\alpha_5\alpha_4$, $f\rho_2^{-1}=\alpha_9\alpha_8\alpha_7$ for $\alpha_1,\dots,\alpha_9,\beta_1,\beta_2,\gamma_1,\gamma_2,\tilde{\rho}_1,\tilde{\rho}_2\in(\Aut(\PP^2)\cup\Aut(\FF_0)\cup\Aut(\FF_2))\cap J$, only $\alpha_2,\alpha_5,\alpha_8, \tilde{\rho}_1,\tilde{\rho}_2$ quadratic (Lemma~\ref{lem dJ}). Lemma~\ref{lem degree 1,2} implies that all triangles of the above diagram are generated by relations in $\mathfrak{G}$ and thus $w(f)w(g)w(h)=w(\alpha_9)\cdots w(\alpha_1)$ in $\mathfrak{G}$. We obtain the $\alpha_i$'s in the claim by merging neighbour automorphisms of $\PP^2$ in the product $\alpha_9\cdots\alpha_1$.
\end{proof}

\section{The Cremona group is isomorphic to $\mathfrak{G}$}\label{4}

In this section we prove Theorem B (Theorem~\ref{amalg}). The main tool will be Lemma~\ref{lem dJ Id} which yields the existence of an injective map $w_J\colon J\rightarrow\mathfrak{G}$ such that $\pi\circ w_J=\Id$ (Corollary~\ref{cor dJ Id}) and enables us to use the result (Theorem~\ref{bla}) of \cite{Bla12}, that $\Bir(\PP^2)$ is isomorphic to the amalgamated product of $\Aut(\PP^2)$ and $J$ along their intersection modulo one relation, for the proof of Theorem B (Theorem~\ref{amalg}).

\begin{Lem}\label{lem dJ Id}
Let $f_1,\dots,f_n\in\left(\Aut(\PP^2)\cup\Aut(\FF_0)\cup\Aut(\FF_2)\right)\cap J$ such that $f_n\cdots f_1=\Id$. Then $w(f_n)\cdots w(f_1)=\Id$ in $\mathfrak{G}$. 
\end{Lem}

\begin{proof}
We can write $w(f_n)\cdots w(f_1)=w(\alpha_{m+1})w(g_m)w(\alpha_m)\cdots w(\alpha_2)w(g_1)w(\alpha_1)$ where $\alpha_i\in\Aut(\PP^2)\cap J$ and $g_i\in\left(\Aut(\FF_0)\cup\Aut(\FF_2)\right)\cap J\setminus\Aut(\PP^2)$ as follows: We put $g_j:=f_i$ if $f_i$ is quadratic,  $\alpha_j:=f_i$ if $f_i$ is linear. Then we proceed by putting $\alpha_j:=\alpha_{i+1}\alpha_i$ ($w(\alpha_j)=w(\alpha_{i+1}\alpha_i)$ by Lemma~\ref{lem degree 1,2}). Proceeding like this we will reach a word where no two consecutive letters both have linear image in $\Bir(\PP^2)$. We then insert $\alpha_j=\Id$ between any two consecutive letters whose both image in $\Bir(\PP^2)$ is quadratic.

We denote by $\Delta_0$ the linear system of lines in $\PP^2$ and define for $i=1,\dots,m$
\[\Delta_i:=(\alpha_ig_{i-1}\cdots g_1\alpha_1)(\Delta_0)\]
which is the linear system of the map $(\alpha_ig_{i-1}\cdots g_1\alpha_1)^{-1}$. We define $d_i:=\deg(\Delta_i)$, which is also the degree of the map $(\alpha_ig_{i-1}\cdots g_1\alpha_1)^{-1}$. Furthermore, we define
\[D:=\max\{d_i\mid i=1,\dots,m\},\qquad N:=\max\{i\mid d_i=D\}\]
If $D=1$, it follows that $m=1$ and $\alpha_1=\Id$.  We can therefore assume that $D>1$ and prove the result by induction over the lexicographically ordered pair $(D,N)$. 

The induction step consists of finding $\tilde{\alpha}_{k+1},\dots,\tilde{\alpha}_1\in\Aut(\PP^2)\cap J$ and $\tilde{g}_1,\dots,\tilde{g}_k\in\left(\Aut(\FF_0)\cup\Aut(\FF_2)\right)\cap J\setminus\Aut(\PP^2)$ such that 
$$w(g_{N+1})w(\alpha_{N+1})w(g_N)=w(\tilde{\alpha}_{k+1})w(\tilde{g}_k)\cdots w(\tilde{g}_1)w(\tilde{\alpha}_{1})$$
and such that the pair $(\tilde{D},\tilde{N})$ associated to the product 
$$\alpha_{m+1}g_m\cdots g_{N+2}(\alpha_{N+2}\tilde{\alpha}_{k+1})\tilde{g}_k\cdots \tilde{g}_1(\tilde{\alpha}_1\alpha_N)g_{N-1}\cdots g_1\alpha_1$$ is strictly smaller than $(D,N)$. 

We look at three cases, depending on the degree of $g_{N+1}\alpha_{N+1}g_N$, and if the degree is three, we look at two sub cases, the "good case" and the "bad case".

If $\deg(g_{N+1}\alpha_{N+1}g_N)=1$, define $\Aut(\PP^2)\ni\tilde{\alpha}:=g_{N+1}\alpha_{N+1}g_N$. It follows from Lemma~\ref{lem degree 1,2}~(1) that $w(g_{N+1})w(\alpha_{N+1})w(g_N)=w(\tilde{\alpha})$ in $\mathfrak{G}$. We replace $g_{N+1}\alpha_{N+1}g_N$ by $\tilde{\alpha}$, which decreases $(D,N)$. 

If $\deg(g_{N+1}\alpha_{N+1}g_N)=2$, it follows from Lemma~\ref{lem degree 1,2}~(2),(3) that there exists $\tilde{\alpha},\tilde{\beta}\in\Aut(\PP^2)\cap J$ and $\tilde{g}\in(\Aut(\FF_0)\cup\Aut(\FF_2))\cap J\setminus\Aut(\PP^2)$ such that $w(g_{N+1})w(\alpha_{N+1})w(g_{N})=w(\tilde{\beta})w(\tilde{g})w(\tilde{\alpha})$. We replace $g_{N+1}\alpha_{N+1}g_N$ by $\tilde{\beta}\tilde{g}\tilde{\alpha}$, which decreases $(D,N)$. 
 
Finally, suppose that $\deg(g_{N+1}\alpha_{N+1}g_N)=3$. By definition of $N$, we have
\[d_{N-1}\leq D,\quad d_N=D,\quad d_{N+1}<D .\] 
"Good case": If $\Delta_{N}$ has a proper simple base-point, it follows from Lemma~\ref{lem deg 3 proper base-point} (with $\Delta=\Delta_{N-1}$) that there exist $\tilde{\alpha}_1,\dots,\tilde{\alpha}_{4}\in\Aut(\PP^2)\cap J$, $\tilde{g}_1,\tilde{g}_2,\tilde{g}_3\in(\Aut(\FF_0)\cup\Aut(\FF_2))\cap J$ such that 
\[w(g_{N+1})w(\alpha_{N+1})w(g_{N})=w(\tilde{\alpha}_{4})w(\tilde{g}_3)\cdots w(\tilde{\alpha}_2)w(\tilde{g}_1)w(\tilde{\alpha}_1)\ \text{in}\ \mathfrak{G}\]
and
\[\deg((\tilde{\alpha}_{i+1}\tilde{g}_i\cdots g_1\tilde{\alpha_1})(\Delta_{N-1}))<\deg(\Delta_{N})=D\]
for $i=1,\dots,4$. Replacing $g_{N+1}\alpha_{N+1}g_N$ by $\tilde{\alpha}_{4}\tilde{g}_3\cdots \tilde{g}_1\tilde{\alpha}_1$ decreases $(D,N)$.

"Bad case": Assume that $\Delta_N$ has no simple proper base-points. Without changing the pair $(D,N)$ we will replace the word $w(\alpha_{m+1})w(g_m)\cdots w(g_1)w(\alpha_1)$ in $\mathfrak{G}$ by an equivalent word $w(\hat{\alpha}_{m+1})w(\hat{g}_m)\cdots w(\hat{g})w(\hat{\alpha}_1)$ satisfying the "good case".

Choose two general points $p_0,q_0$ in $\PP^2$ and write $p_1=(\alpha_2g_1\alpha_1)(p_0)$, $q_1=(\alpha_2g_1\alpha_1)(q_0)$ and $p_i=(\alpha_{i+1}g_i)(p_{i-1})$, $q_i=(\alpha_{i+1}g_i)(q_{i-1})$ for $i=2,\dots,m$. Note that $p_m=p_0$ and $q_m=q_0$ because $\alpha_{m+1}g_m\cdots g_1\alpha_1=\Id$.

For $i=0,\dots,m$, we denote by $\beta_i\in\Aut(\PP^2)$ an element sending $[1:0:0]$, $[0:1:0]$, $[0:0:1]$ respectively onto $[1:0:0],p_i,q_i$ (this is possible, because we took $p_0,q_0$ general), and write $\tau_i:=\beta_i\sigma_3(\beta_i)^{-1}$, which is a quadratic de Jonquières involution having base-points $[1:0:0],p_i,q_i$. We choose $\beta_m=\beta_0$ and then have $\tau_m=\tau_0$. 

By Lemma~\ref{lem square dJ} the maps $\tau_1(\alpha_2g_1\alpha_1)\tau_0^{-1}$, $\tau_i(g_i\alpha_i)\tau_{i-1}^{-1}$ are quadratic de Jonquières and there exist $\gamma_i,\delta_i\in\Aut(\PP^2)\cap J$, $\hat{g}_i\in\{\sigma_2,\sigma_3,\tau_{12}\sigma_2\tau_{12}\}$ such that
\begin{align*}
w(\beta_1)w(\sigma_3)w(\beta_1^{-1})w(\alpha_2)w(g_1)w(\alpha_1)w(\beta_0)w(\sigma_3)w(\beta_0^{-1})=&w(\delta_1)w(\hat{g}_1)w(\gamma_1)\\
w(\beta_i)w(\sigma_3)w(\beta_i^{-1})w(\alpha_{i+1})w(g_i)w(\beta_{i-1})w(\sigma_3)w(\beta_{i-1}^{-1})=&w(\delta_i)w(\hat{g}_i)w(\gamma_i)
\end{align*}
 for $i=1,\dots,m$. We get the following diagram
\[\xymatrix{ \ar[r]^{\alpha_2g_1\alpha_1}\ar[d]_{\tau_0}&\ar[r]^{\alpha_3g_2}\ar[d]^{\tau_1}&\ar[d]^{\tau_2}&\cdots&&\ar[r]^{\alpha_{i+1}g_i}\ar[d]_{\tau_{i-1}}&\ar[d]^{\tau_i}&\cdots&&\ar[r]^{\alpha_{m+1}g_m}\ar[d]_{\tau_{m-1}}&\ar[d]^{\tau_m} \\ 
\ar[r]_{\delta_1\hat{g}_1\gamma_1}&\ar[r]_{\delta_2\hat{g}_2\gamma_2}&&\cdots&&\ar[r]_{\delta_i\hat{g}_i\gamma_i}&&\cdots&&\ar[r]_{\delta_m\hat{g}_m\gamma_m}&}\]
where each square in the diagram corresponds to a relation in $\mathfrak{G}$, making the whole diagram correspond to a relation in $\mathfrak{G}$. 
Therefore, writing $\tilde{\alpha}_i:=\delta_i\gamma_{i-1}$ for $i=2,\dots,m$, $\tilde{\alpha}_{m+1}:=\delta_m$, $\tilde{\alpha}_1:=\gamma_1$, the equality
\[w(\alpha_{m+1})w(g_m)\cdots w(g_1)w(\alpha_1)=w(\hat{\alpha}_{m+1})w(\hat{g}_m)w(\hat{\alpha}_m)\cdots w(\hat{\alpha}_2)w(\hat{g}_1)w(\hat{\alpha}_1)\] 
holds in $\mathfrak{G}$. We replace $\alpha_{m+1}g_m\cdots g_1\alpha_1$ by $\hat{\alpha}_{m+1}\hat{g}_m\hat{\alpha}_m\cdots\hat{\alpha}_2\hat{g}_1\hat{\alpha}_1$. 

For $i=1,\dots,m$, call $\hat{\Delta}_i:=(\hat{\alpha}_i\hat{g}_{i-1}\cdots \hat{g}_1\hat{\alpha}_1)(\Delta_0)$, which is the linear system of the map $(\hat{\alpha}_i\hat{g}_{i-1}\cdots\hat{g}_1\hat{\alpha}_1)^{-1}$, and denote by $\hat{d}_i$ its degree. Using Remark~\ref{deg dJ} we get $\deg(\hat{\alpha}_i\hat{g}_{i-1}\cdots\hat{g}_1\hat{\alpha}_1)=\deg(\alpha_ig_{i-1}\cdots g_1\alpha_1)$ for each $i$, thus $\hat{d}_{i}=d_i$ for $i=1,\dots,m$. Therefore, the replacement does not change the pair $(D,N)$, i.e. $(\hat{D},\hat{N})=(D,N)$. 
 
It remains to show that $\hat{\alpha}_{n+1}\hat{g}_n\cdots\hat{g}_1\hat{\alpha}_1$ satisfies the "good case", i.e. that $\hat{\Delta}_N$ has a simple proper base-point.

Since $d_{N+1}<D$, it follows from Remark~\ref{deg dJ} that $d_{N+1}=D-1$ and that all the base-points of $g_{N+1}$ are base-points of $\Delta_N$. Since the base-point of $\tau_N$ are general it follows from Remark~\ref{deg dJ} that each point in $(\tau_N)^{\bullet}(\Bp(g_N))$ is a base-point of $\hat{\Delta}_N$. Lemma~\ref{lem square dJ} states that $(\tau_N)^{\bullet}(\Bp(g_N))=\Bp(\hat{g}_N)$, and $\hat{g}_N\in\{\sigma_2,\sigma_3,\tau_{12}\sigma_2\tau_{12}\}$ has a simple proper base-point. Hence $\hat{\Delta}_N$ has a simple proper base-point.
\end{proof}

\begin{Cor}\label{cor dJ Id}
Let $\alpha_1,\dots,\alpha_n,\beta_1,\dots,\beta_m\in\Aut(\PP^2)\cup\Aut(\FF_0)\cup\Aut(\FF_2)$ de Jonquières such that $\alpha_n\cdots\alpha_1=\beta_m\cdots\beta_1$. Then 
\[w(\alpha_n)\cdots w(\alpha_1)=w(\beta_m)\cdots w(\beta_1).\] 
In particular there exists a homomorphism $w_J\colon J\rightarrow\mathfrak{G}$ which sends $\alpha_n\cdots\alpha_1 $ onto $w(\alpha_n)\cdots w(\alpha_1)$ and $\pi\circ w_J=\Id$, i.e. $w_J$ is injective. 
\end{Cor}

\begin{proof}
The claim follows from applying Lemma~\ref{lem dJ Id} to $\beta_1^{-1}\cdots\beta_m^{-1}\alpha_n\cdots\alpha_1$. 
\end{proof}

\begin{Prop}[\cite{Bla12}]\label{bla}
The group $\Bir(\PP^2)$ is isomorphic to $$\left(\Aut(\PP^2)\ast_{\Aut(\PP^2)\cap J}J\right)/\langle\tau_{12}\sigma_3\tau_{12}\sigma_3\rangle,$$ 
the amalgamated product of $\Aut(\PP^2)$ and $J$ along their intersection and divided by the relation $\tau_{12}\sigma_3=\sigma_3\tau_{12}$, where $\tau_{12}([x:y:z])=[y:x:z]$ .
\end{Prop}

\begin{Rmk}
In $\Bir(\PP^2)$, the three relations 
\begin{enumerate}
\item\label{rel1}$\tau_{12}\sigma_3\tau_{12}\sigma_3=\Id$ 
\item\label{rel2}$\tau_{13}\sigma_3\tau_{13}\sigma_3=\Id$
\item\label{rel3}$\tau_{23}\sigma_3\tau_{23}\sigma_3=\Id$ 
\end{enumerate}
hold. Choosing two of them, the remaining relation of the three is generated by the chosen two. Relation~\ref{rel3} is a relation holding in $J$. Thus it suffices to impose relation~\ref{rel1} or \ref{rel2} in Theorem~\ref{bla}. However, since $\tau_{12},\sigma_3\in\Aut(\FF_0)$, relation~\ref{rel1} holds in $\Aut(\FF_0)$, so in particular it holds in the generalised amalgamated product of $\Aut(\PP^2),\Aut(\FF_0),\Aut(\FF_2)$ along all their pairwise intersections. 

It is not clear at all whether $J$ embeds into the generalised amalgamated product of $\Aut(\PP^2),\Aut(\FF_0),\Aut(\FF_2)$ along all their pairwise intersections, so it is à priori not clear whether one of the relations~\ref{rel2}, \ref{rel3} holds in there. Thus we need to impose one of them. 
\end{Rmk}

\begin{Thm}[(Theorem B)]\label{amalg} 
The group $\Bir(\PP^2)$ is isomorphic to $\mathfrak{G}$, the generalised amalgamated product of $\Aut(\PP^2)$, $\Aut(\FF_0)$, $\Aut(\FF_2)$ along all the pairwise intersections modulo the relation $\tau_{13}\sigma_3\tau_{13}\sigma_3$ where $\tau_{13}([x:y:z])=[z:y:x]$. 
\end{Thm}

\begin{proof}
By Corollary~\ref{cor dJ Id} there exists $w_J\colon J\rightarrow\mathfrak{G}$ such that $\pi\circ w_J=\Id$, and $w$ and $w_J$ coincide on $\Aut(\PP^2)\cap J$. Thus the following diagram commutes
\[\xymatrix{ \mathfrak{G}&J\ar[l]_{w_J}\\ \Aut(\PP^2)\ar[u]^w&\Aut(\PP^2)\cap J\ar@{_{(}->}[l]_{\iota_1}\ar@{^{(}->}[u]^{\iota_2} }\]
where $\iota,\iota_1$ are the canonical inclusion maps. The universal property of the amalgamated product implies the existence of a unique homomorphism $\varphi:\Aut(\PP^2)\ast_{\Aut(\PP^2)\cap J}J\rightarrow\mathfrak{G}$ such that the following diagram commutes:
\[\xymatrix{\mathfrak{G}&&\\ &\Aut(\PP^2)\ast_{\Aut(\PP^2)\cap J}J\ar[ul]^{\exists !}_{\varphi}&J\ar[l]\ar@/_1pc/[ull]^{w_J}\\ &\Aut(\PP^2)\ar[u]\ar@/^1pc/[luu]^w&\Aut(\PP^2)\cap J\ar@{_{(}->}[l]_{\iota}\ar@{^{(}->}[u]^{\iota}}.\]
By Proposition~\ref{bla}, $\Bir(\PP^2)$ is isomorphic to $\Aut(\PP^2)\ast_{\Aut(\PP^2)\cap J}J$ modulo the relation $\sigma_3\tau_{12}=\tau_{12}\sigma_3$, where $\tau_{12}([x:y:z])=[y:x:z]$. Since $\tau_{12},\sigma_3\in\Aut(\FF_0)$, the relation $\sigma_3\tau_{12}=\tau_{12}\sigma_3$ also holds in $\Aut(\FF_0)$ and hence in $\mathfrak{G}$. Thus, the homomorphism $\varphi$ induces a homomorphism $\bar{\varphi}:(\Aut(\PP^2)\ast_{\Aut(\PP^2)\cap J}J)/\langle\sigma_3\tau_{12}\sigma_3\tau_{12}\rangle\longrightarrow\mathfrak{G}$. By construction, $\bar{\varphi}$ and the canonical homomorphism $\pi:\mathfrak{G}\rightarrow\Bir(\PP^2)$ are inverse to each other. 
\end{proof}

\section{The Cremona group is compactly presented}\label{5}

In this section, we restrict to case $k=\C$ and show that $\Bir(\PP^2)$ is compactly presented using Theorem~\ref{amalg} (Theorem B). 

Being compactly presented is a notion reserved for Hausdorff topological groups and we consider $\Bir(\PP^2)$ endowed with the Euclidean topology as constructed in \cite[Section~5]{BF13} which makes $\Bir(\PP^2)$ a Hausdorff topological group \cite[Theorem~3]{BF13}, which is not locally compact \cite[Lemma~5.15]{BF13}. 

\begin{Def}\label{def cp} Let $G$ be a group.
\begin{enumerate}
 \item A {\em presentation} $\langle S\mid R\rangle$  of $G$ is a triple made up of a set $S$, an epimorphism $\pi:F_S\twoheadrightarrow G$ of the free group on $S$ onto $G$, a subset $R$ of $F_S$ generating $\ker(\pi)$ as a normal subgroup. The {\em relations} of the presentation are the elements of $\ker(\pi)$ and the elements of $R$ the {\em relators} (or {\em generating relations}) of the presentation. 
\item A {\em bounded presentation} of $G$ is a presentation $\langle S\mid R\rangle$ of $G$ with $R$ a set of relators of bounded length.
\item Let $G$ be a Hausdorff topological group. A {\em compact presentation} of $G$ is a presentation $\langle S\mid R\rangle$ of $G$ with $S$ a compact subset of $G$ and $R$ a set of relators of bounded length. 
We say that $G$ is {\em compactly presented by} $S$ if $G$ is given by a compact presentation $\langle S\mid R\rangle$. We also say that $G$ is {\em compactly presented} if $G$ is compactly presented by some subset.
\end{enumerate}
\end{Def}

\begin{Lem}\label{cps}
\begin{enumerate}
\item\label{62} Let $G$ be a group and $S_1,S_2\subset G$ generating subsets. If $S_1^m\subset S_2^n\subset S_1^{m'}$ for some $m,n,m'\in\N$ then $G$ is boundedly presented by $S_1$ if and only if it is boundedly presented by $S_2$. 
\item\label{64} Any connected topological group is generated by any neighbourhood of $1$.
\item\label{61} If $G$ is a locally compact Hausdorff topological group having only finitely many connected components it is compactly presented.
\item\label{63} If $G$ is a locally compact Hausdorff topological group that is compactly presented then it is compactly presented by all its compact generating subsets.
\item\label{65} Let $G$ be a locally compact topological group with finitely many connected components $G_0,\dots,G_n$, where $1\in G_0$. For each $i$ choose some $g_i\in G$ such that $G_i=g_iG_0$. Then $G$ is generated by any compact neighbourhood of $1$ and $g_1,\dots,g_n$. In particular, it is compactly presented by any compact neighbourhood of $1$ and $g_1,\dots,g_n$.
\end{enumerate}
\end{Lem}

\begin{proof}
\ref{62} is proved in \cite[Lemma~7.A.9]{CH14} and \cite[Lemma~2.6]{Cor} and \ref{61} in \cite[Satz~3.2]{Ab} (see also \cite[\S8.A]{CH14}).

\ref{64}: Let $U\subset G$ be an open neighbourhood of $1$. Then the subgroup $H$ of $G$ generated by $U$ is open because $H=\bigcup_{h\in H}hU$. It is also closed because $G\setminus H=\bigcup_{g\in G\setminus H}gH$, which is a open set.  

\ref{63}: If $G$ is compactly generated by a compact set $S$ and $K\subset G$ is a compact set then $K\subset S^n$ for some large $n$. 
This follows from the fact that any locally compact topological group is a Baire space and that $S$ is compact. The claim now follows from \ref{62}. 

\ref{65}: Let $K\subset G_0$ be a compact neighbourhood of $1$. By \ref{64}, $K$ generates $G_0$ and thus the compact set $K\cup\{g_1,\dots,g_n\}$ generates $G$. By \ref{61} and \ref{63} the locally compact group $G$ is compactly presented by $K\cup\{g_1,\dots,g_n\}$. 
\end{proof}

\begin{Rmk}\label{rmk zariski eucl topology}
Any irreducible algebraic variety over $\C$ is connected with respect to the Euclidean topology \cite[Chapter~XII, Proposition~2.4]{SGA1}. Any linear algebraic subgroup of $\Bir(\PP^2)$ has finitely many irreducible components in the Zariski-topology, which are exactly the cosets of the component containting $1$. Thus they are the connected components in the Zariski topology and hence also the connected components in the Euclidean topology. 

Furthermore, any linear algebraic subgroup of $\Bir(\PP^2)$ is a closed subset of $\Bir(\PP^2)_{\leq d}$ for some $d\in\N$ \cite[Lemma~2.19]{BF13}, which is a locally compact set \cite[Lemma~5.4]{BF13}. Hence any linear algebraic subgroup of $\Bir(\PP^2)$ is locally compact and therefore satisfies the conditions of Lemma~\ref{cps}~\ref{65}. 
\end{Rmk}

\begin{Rmk}\label{rmk grp}
Any algebraic subgroup of $\Bir(\PP^2)$ is a linear algebraic group \cite[Théorème~2]{Bla09}. The Euclidean topology on these groups is exactly the topology inherited from the Euclidean topology on $\Bir(\PP^2)$ \cite[Proposition~5.11]{BF13}.

The groups $\Aut(\PP^2)=\mathrm{PGL}_3(\C)$,  $\Aut(\FF_0)$ and $\Aut(\FF_2)$ are linear algebraic subgroups of $\Bir(\PP^2)$ (Lemma~\ref{cor auto linear}), and thus locally compact by Remark~\ref{rmk zariski eucl topology}.
\end{Rmk}

\begin{Cor}\label{cor cp}
\begin{enumerate}
\item The group $\Aut(\PP^2)$ is compactly presented by any compact neighbourhood of $1$. 
\item The group $\Aut(\FF_0)$ is compactly presented by the union of the linear map $\tau_{12}\colon[x:y:z]\rightarrow[y:x:z]$ and any compact neighbourhood of $1$.
\item The group $\Aut(\FF_2)$ is compactly presented by any compact neighbourhood of $1$.
\end{enumerate}
\end{Cor}

\begin{proof}
The groups $\Aut(\PP^2),\Aut(\FF_0),\Aut(\FF_2)$ are linear algebraic groups and locally compact by Remark~\ref{rmk grp}.

The group $\Aut(\PP^2)=\mathrm{PGL}_3(k)$ is irreducible, hence connected (Remark~\ref{rmk zariski eucl topology}), and the group $\Aut(\FF_2)$ is connected by Lemma~\ref{cor auto linear}~\ref{43}. By Lemma~\ref{cor auto linear}~\ref{42}, the group $\Aut(\FF_0)$ has two connected components, namely $\Aut(\FF_0)^0$ containing the identity element and $\tau_{12}\Aut(\FF_0)^0$. The claim now follows from Remark~\ref{rmk zariski eucl topology} and Proposition~\ref{cps}~\ref{65}.
\end{proof}

Using Corollary~\ref{cor cp} and the fact that $\Bir(\PP^2)$ is isomorphic to the generalised amalgamated product of $\Aut(\PP^2)$, $\Aut(\FF_0)$, $\Aut(\FF_2)$ along their pairwise intersection divided by one relation (Theorem~\ref{amalg}) we prove that $\Bir(\PP^2)$ is compactly presentable. 

\begin{Lem}\label{amalg product generators}
Let  $G$ be a group, $n\ge 2$ be an integer, and $G_1,\dots,G_n\subset G$, be subgroups of $G$  such that the following hold:
\begin{enumerate}
\item
The group $G$ admits the presentation $G=\left\langle \bigcup\limits_{i=1}^n G_i \mid R_G\right\rangle$, where $R_G$ is the set of all relators of the form $ab=c$, where $a,b,c\in G_i$ for some $i\in \{1,\dots,n\}$. 
\item For $i=1,\dots,n$, there exists a presentation $\langle K_i\mid R_i\rangle$ of $G_i$ such that for any subset $I\subset \{1,\dots,n\}$ the set $\bigcap\limits_{i\in I} K_i$ generates $\bigcap\limits_{i\in I} G_i$.
\end{enumerate}
Then, $G$ admits the presentation $G=\left\langle \bigcup\limits_{i=1}^n K_i \mid  \bigcup\limits_{i=1}^n R_i \right\rangle$.
\end{Lem}

\begin{proof}

Denote by $F_G$ the free group generated by $\bigcup\limits_{i=1}^n G_i$ and by $F_K$ the free group generated by $K=\bigcup\limits_{i=1}^n K_i$; we view $F_K$ as a subgroup of $F_G$.

The natural group homomorphism $\pi \colon F_K\to G$ is surjective, because $G$ is generated by  $\bigcup\limits_{i=1}^n G_i$ and each $G_i$ is generated by $K_i$. Moreover, each set of relators $R_i$ corresponds to a subset of $\ker(\pi)$. It remains to see that $\ker \pi$ is contained in the normal subgroup generated by $\bigcup\limits_{i=1}^nR_i$. 

We take an element in $\ker(\pi)$, which in $F_K$ is a word
$$w=s_1s_2\dots s_m$$
such that each $s_i$ belongs to $K$ and $s_1\cdots s_m=1$ in $G$. Because $G$ admits the presentation $G=\left\langle \bigcup\limits_{i=1}^n G_i \mid R_G\right\rangle$, we can write $w$ in $F_G$ as a product
$$w=a_1r_1a_1^{-1}a_2r_2a_2^{-1}\dots a_lr_la_l^{-1}$$
where all the $a_i,r_i$ are elements of  $F_G$ and $r_i\in R_G$, which means by definition of $R_G$ that $r_i=a_ib_ic_i$ for $a_i,b_i,c_i\in G_{j(i)}$, i.e. each $r_i$ is a word in elements of $G_{j(i)}$.

The word $w$ is equal in $F_G$ to a reduced word, whose letters are elements of $K$ because $w\in F_K$. Hence, each  $g=\bigcup\limits_{i=1}^n G_i\setminus K$ which appears in the word $a_1r_1a_1^{-1}\dots a_lr_la_l^{-1}$ disappears after the reduction. We can thus replace each occurrence of $g$ with any chosen element of $F_G$ and do not change the value of the word. We do this in the following way:  if $g\in G_i\setminus K_i$, we replace then $g$ with a word with letters in $K_i$, which belongs to $\pi^{-1}(\pi(g))$ (this is possible since $K_i$ generates $G_i$). If $g$ belongs to more than one of the $G_i$ we can moreover assume that the letters of the word also belong to these $K_i$, because of the second hypothesis.

After this replacement, we obtain an equality in $F_K$
$$s_1\dots s_m=b_1t_1b_1^{-1}b_2t_2b_2^{-1}\dots b_lt_lb_l^{-1},$$
where each $t_i$ is a word with letters in $K_{j(i)}$, such that $\pi(t_i)=1$. For each $i=1,\dots,n$ denote by $F_{K_i}$ the free group generated by $K_i$ and by $\pi_i:F_{K_i}\rightarrow G_i$ the natural group homomorphism onto $G_i$ whose kernel is generated by $R_i$. We consider $F_{K_i}$ as subgroup of $F_K$, which means that $\pi_i=\pi|_{F_{K_i}}$ and hence $\ker(\pi_i)=\ker(\pi)\cap F_{K_i}$. Therefore $\pi_i(t_i)=1$ and thus $t_i$ is a product of conjugates of $R_{j(i)}$. This yields the result.
\end{proof}

\begin{Cor}\label{cor cpt pres}
Let $K\subset\Aut(\PP^2)$, $K_0\subset\Aut(\FF_0)$, $K_2\subset\Aut(\FF_2)$ be compact neighbourhoods of $1$ in the respective groups. Then $\Bir(\PP^2)$ is compactly presented by $K\cup K_0\cup K_2\cup\{\tau_{12}\}$.
\end{Cor}

\begin{proof}
Lemma~\ref{amalg product generators} yields that the union of any compact generating sets of $\Aut(\PP^2)$, $\Aut(\FF_0)$, $\Aut(\FF_2)$ giving a compact presentation of the respective groups yields a compact presentation of $\mathfrak{G}$, the generalised amalgamated product of $\Aut(\PP^2)$, $\Aut(\FF_0)$, $\Aut(\FF_2)$ along their pairwise intersection divided by one relation. Such compact generating sets are given by Corollary~\ref{cor cp}: Any compact neighbourhood of 1 of the groups $\Aut(\PP^2)$ and $\Aut(\FF_2)$ respectively, and the union of $\tau_{12}$ and any compact neighbourhood of 1 in $\Aut(\FF_0)$. Since $\Bir(\PP^2)$ and $\mathfrak{G}$ are isomorphic (Theorem~\ref{amalg}), the claim follows. 
\end{proof}

\begin{Cor}\label{cor cpt pres 2}
Let $K\subset\Aut(\PP^2)$, $K_0\subset\Aut(\FF_0)$, $K_2\subset\Aut(\FF_2)$ be compact neighbourhoods of $1$ in the respective groups. Then $\Bir(\PP^2)$ is compactly presented by $K\cup K_0\cup K_2$.
\end{Cor}

\begin{proof}
We define $S_1:=K\cup K_0\cup K_2$ and $S_2:=K\cup K_0\cup K_2\cup\{\tau_{12}\}$. The set $S_2$ generates $\Bir(\PP^2)$ by Corollary~\ref{cor cpt pres}. The set $K$ generates $\Aut(\PP^2)$ (Corollary~\ref{cor cp}), hence there exists $n\in\N$ such that $\tau_{12}\in K^n$. It follows that also $S_2$ generates $\Bir(\PP^2)$ and moreover that $S_1\subset S_2\subset (S_1)^n$. The claim now follows from Lemma~\ref{cps}~\ref{62} and Corollary~\ref{cor cpt pres}. 
\end{proof}

Lemma~\ref{cps}~\ref{62} and Corollary~\ref{cor cpt pres 2} imply that to prove Theorem A (Corollary~\ref{cor cpt pres 3}) we only need to check that for any compact neighbourhood $K\subset\Aut(\PP^2)$ of $1$ there exist $K_i\subset\Aut(\FF_i)$, $i=0,2$, compact neighbourhoods of 1 and integers $m,m',n\in\N$ such that $(K\cup\{\sigma_3\})^m\subset(K\cup K_0\cup K_2)^n\subset(K\cup\{\sigma_3\})^{m'}$. 

\begin{Lem}\label{lem open set}
Let $K\subset\Aut(\PP^2)$ be a compact neighbourhood of $1$. Then there exists $N\in\N$ such that $(K\cup \{\sigma_3\})^N$ contains compact neighbourhoods of $1$ in $\Aut(\FF_i)$, for $i=0,2$.
\end{Lem}

\begin{proof}
Let $\mathcal{A}_0=\Aut(\FF_0)^0\cap\Aut(\PP^2)$ and $\mathcal{A}_2=\Aut(\FF_2)\cap\Aut(\PP^2)$, which are connected algebraic subgroups of $\Aut(\FF_0)$, $\Aut(\FF_2)$ respectively (Lemma~\ref{cor auto linear}). For $i=0,2$, the set $K_i=K\cap \mathcal{A}_i$ is a compact neighbourhood of $1$ in $\mathcal{A}_i$. Corollary~\ref{cor cp} implies that $\mathcal{A}_i=\bigcup_{n\in\N}(K_i)^n$, for $i=0,2$. It follows that 
\[\mathcal{A}_0\sigma_3\mathcal{A}_0=\bigcup_{n\in\N}(K_0)^n\sigma_3(K_0)^n,\qquad \mathcal{A}_2\sigma_2\mathcal{A}_2=\bigcup_{n\in\N}(K_2)^n\sigma_2(K_2)^n.\]
The sets $\mathcal{A}_0\sigma_3\mathcal{A}_0$ and $\mathcal{A}_2\sigma_2\mathcal{A}_2$ are Zariski-open subsets of $\Aut(\FF_0)$, $\Aut(\FF_2)$ respectively (Lemma~\ref{cor auto linear}), and are thus locally compact and hence Baire spaces. There exists then some $m\in \mathbb{N}$ such that  
 that $(K_0)^{m}\sigma_3(K_0)^{m}$ and $(K_2)^{m}\sigma_2(K_2)^{m}$ have non-empty interior in $\mathcal{A}_0\sigma_3\mathcal{A}_0$ and $\mathcal{A}_2\sigma_2\mathcal{A}_2$, and thus in $\Aut(\FF_0)$ and $\Aut(\FF_2)$ respectively.

Since $(K_i)^{m}\sigma_j(K_i)^{m}\subset(K_i\cup\{\sigma_j\})^{2m+1}$, the sets $(K_0\cup\{\sigma_3\})^{2m+1}$ and $(K_2\cup\{\sigma_2\})^{2m+1}$ also have non-empty interior in $\Aut(\FF_0)$, $\Aut(\FF_2)$ respectively. Since $(K_i)^{-1}\subset K_i^{m_i}$ for some big $m_i$ and $(\sigma_j)^{-1}=\sigma_j$, the sets $(K_0\cup\{\sigma_3\})^{m'}$ and $(K_2\cup\{\sigma_2\})^{m'}$ are neighbourhoods of $1$ in the corresponding groups for some $m'$ big enough. Since $\Bir(\PP^2)$ is generated by $K\cup \{\sigma_3\}$ (by the Noether-Castelnuovo theorem), we find $m''$ such that $\sigma_2\in (K\cup \{\sigma_3\})^{m''}$. A suitable power of $K\cup \{\sigma_3\}$ contains thus $(K_0\cup\{\sigma_3\})^{m'}$ and $(K_2\cup\{\sigma_2\})^{m'}$.
\end{proof}

\begin{Cor}[(Theorem A)]\label{cor cpt pres 3}
Let $K\subset\Aut(\PP^2)$ be a compact neighbourhood of $1$. Then $\Bir(\PP^2)$ is compactly presented by $K\cup\{\sigma_3\}$.
\end{Cor}

\begin{proof}According to Lemma~\ref{lem open set}, there exists $N\in\N$ and compact neighbourhoods  $K_0,K_2$ of $1$ in $\Aut(\FF_0)$ and $\Aut(\FF_2)$ respectively, such that $(K\cup \{\sigma_3\})^N$ contains $K_0\cup K_2$.

We define $S_1:=K\cup\{\sigma_3\}$ and $S_2:=K\cup K_0 \cup K_2$. 
Because $\sigma_3\in \Aut(\FF_0)^0$ and $\Aut(\FF_0)^0$ is compactly generated by $K_0$ (Lemma~\ref{cps}~\ref{64}) there exists $M\in \mathbb{N}$ such that $\sigma_3\in(K_0)^M$. It follows that $S_1\subset (S_2)^M$. Since $S_2\subset (S_1)^N$, we have  $S_1\subset (S_2)^M\subset (S_1)^{MN}$. The group $\Bir(\PP^2)$ being compactly presented by $S_2$ (Corollary~\ref{cor cpt pres 2}), it is also compactly presented by $S_1$ (Lemma~\ref{cps}~\ref{62}).
\end{proof}

\vskip\baselineskip
\address{
   Susanna Zimmermann\\
   Mathematisches Institut\\ 
   Universit\"at Basel\\ Spiegelgasse 1\\ 4051 Basel, Switzerland}\\
   \email{susanna.zimmermann@unibas.ch}

\end{document}